\documentclass[11pt]{article}

\usepackage{algorithm}
\usepackage{algpseudocode}
\usepackage{fancyvrb}
\usepackage{makeidx}
\usepackage{amsmath,amssymb, amsthm}
\usepackage{latexsym}
\usepackage[toc,page]{appendix}
\usepackage{graphicx}
\usepackage{multirow}
\usepackage{bbm}
\usepackage{amsfonts}
\usepackage{mathtools}
\usepackage{textcomp}
\usepackage{hyperref}
\usepackage{verbatimbox}
\usepackage{subcaption}
\usepackage{graphicx}
\usepackage{array}
\usepackage{booktabs}
\usepackage{textcomp,array}
\usepackage{tikz}
\usepackage{enumerate}
\usepackage{tablefootnote}
\usetikzlibrary{positioning}

\DeclareMathOperator*{\Argmax}{Argmax}
\graphicspath{ {images/} }
\usepackage{url}

\makeatletter
\def\url@leostyle{%
  \@ifundefined{selectfont}{\def\UrlFont{\sf}}{\def\UrlFont{\small\ttfamily}}}
\makeatother
\newtheorem{assumption}{Assumption}
\newtheorem{definition}{Definition}
\newtheorem{lemma}{Lemma}
\newtheorem{proposition}{Proposition}

\newtheorem{remark}{Remark}
\newtheorem{theorem}{Theorem}

\DeclareMathOperator*{\argmin}{argmin} 
 
\newcommand{\N}{\mathbb{N}}

\newcommand{\R}{\mathbb{R}}

\newcommand{\cmark}{\ding{51}}%
\usepackage{pifont}
\newcommand{\xmark}{\ding{55}}%

\begin{document}
\title{Smoothed Proximal Lagrangian Method for Nonlinear Constrained Programs
\footnote{Three authors contributed equally to this work and are listed in alphabetic order.}}
\author{
Wenqiang Pu
\thanks{Shenzhen Research Institute of Big Data, The Chinese University of Hong Kong (Shenzhen), China.}
\and Kaizhao Sun
\thanks{DAMO Academy, Alibaba Group US, Bellevue, WA, USA.}
\and Jiawei Zhang
\thanks{Massachusetts Institute of Technology, Cambridge, MA, USA.}
}
\date{\today}
\maketitle

\begin{abstract}
    This paper introduces a smoothed proximal Lagrangian method for minimizing a nonconvex smooth function over a convex domain with additional explicit convex nonlinear constraints. Two key features are 1) the proposed method is single-looped, and 2) an first-order iteration complexity of $\mathcal{O}(\epsilon^{-2})$ is established under mild regularity assumptions. The first feature suggests the practical efficiency of the proposed method, while the second feature highlights its theoretical superiority. Numerical experiments on various problem scales demonstrate the advantages of the proposed method in terms of speed and solution quality. 
\end{abstract}

\section{Introduction}
In this paper, we consider the following optimization problem 
\begin{align}\label{eq: nlp}
	\min_{x \in X} \{ f(x) :~ h(x) \leq 0 \},
\end{align}
where $x\in \R^n$ are decision variables, $X \subseteq \R^n$ is a convex and compact set that is easy to project onto,  the objective function $f:\R^n \rightarrow \R$ is continuously differentiable and possibly nonconvex, and the mapping $h:\R^{n} \rightarrow \R^m$ consists of $m$ convex and continuously differentiable functions. {We further assume the analytical form of $X$ \eqref{eq: X} for the purpose of theoretical analysis, which can be described as}
\begin{align}\label{eq: X}
	X: = \{x\in \R^n:~ g(x) \leq 0\},
\end{align}
where $g:\R^n \rightarrow \R^l$ consists of $l$ convex and continuously differentiable functions. Our proposed algorithm queries zero/first-order oracles of $f$ and $h_i$'s for $i \in [m] := \{1, \cdots, m\}$ and projection oracles of $X$.

In this paper, we propose a \textit{practically simple} while \textit{theoretically superior} algorithm for problem \eqref{eq: nlp}. Define the proximal Lagrangian function associated with problem \eqref{eq: nlp} as 
\begin{align}
	K(x, z, y) := f(x) + \langle y, h(x) \rangle + \frac{p}{2}\|x - z\|^2, 	
\end{align}
where $(x, z, y) \in \R^n \times \R^n \times \R^m$ and $p > 0$. 
We introduced the Smoothed Proximal Lagrangian Method (SPLM) in Algorithm \ref{Alg:SProxALM}. 

\begin{algorithm}[H]
	\caption{Smoothed Proximal Lagrangian Method (SPLM)} \label{Alg:SProxALM}
\begin{algorithmic}[1]
\State {\textbf{Input:} problem data $(f, h, X)$;}
\State {\textbf{Initialize:} positive parameters $(\alpha, \beta, c, p, B)$, $x^0 \in X$, $z^0 \in X$, and $y^0 \in Y := [0, B]^m$; } \label{alg: init}
\For{$t=0,1,2,\ldots,$} 
\State{$x^{t+1} = \Pi_{X}(x^t - c\nabla_x K(x^t, z^t, y^{t}))$;} \label{alg: primal}
\State{$y^{t+1} = \Pi_{Y}(y^t + \alpha h(x^{t+1}))$;} \label{alg: dual}
\State{$z^{t+1} = z^t + \beta(x^{t+1}-z^t)$; }\label{alg: prox}
\EndFor
\end{algorithmic}
\end{algorithm}

With proper initialization in line \ref{alg: init}, the proposed method updates the primal and dual variables via projected gradient steps as in line \ref{alg: primal} and line \ref{alg: dual}, respectively, and then perform another proximal update in line \ref{alg: prox}. Hence our SPLM is a \textit{single-looped} algorithm invoking only \textit{first-order} oracles of the problem data. A single-looped algorithm can be preferable over algorithms with nested loops \cite{li2021rate, lin2019inexact, melo2020almfullstep} from an implementation point view, as determining the stopping criteria of inner loops can be complex (than needed) in theory and laborious in practice. 

A key feature in Algorithm \ref{Alg:SProxALM} is a projection of dual variables onto a bounded hypercube $Y = [0, B]^m$ in line \ref{alg: dual} for some $B > 0$. This technique has been used in the convergence analysis of augmented Lagrangian methods (ALM) \cite{andreani2008augmented, bertsekas2014constrained} as well as in various applications \cite{sun2021two, sun2029two} to keep dual iterates bounded, but usually feasibility is ensured only when a penalty parameter goes to infinity. However, it is well-known that penalty and ALM-type methods inevitably suffer from numerical issues as the penalty parameter increases. Instead, we shows that in the context of \eqref{eq: nlp}, a constant value of $p$ suffices for the convergence of Algorithm \ref{Alg:SProxALM}, which, to our knowledge, is new to the literature.

Assuming a Slater-like condition and a certain full-rank condition on stationary solutions of \eqref{eq: nlp}, our main result says that, with properly chosen parameters $(\alpha, \beta, c, p, B)$, the proposed method finds an $\epsilon$-stationary point in $\mathcal{O}(\epsilon^{-2})$ iterations, where $\epsilon>0$ measures the violation of feasibility and stationarity. This result significantly improves the existing $\tilde{\mathcal{O}}(\epsilon^{-2.5})$ and $\tilde{\mathcal{O}}(\epsilon^{-3})$ complexity upper bounds\footnote{The $\tilde{\mathcal{O}}$ notation hides the logarithmic dependency on $\epsilon^{-1}$, usually resulting from double-looped implementations. } in the literature \cite{kong2023iteration, lin2019inexact, melo2020almfullstep, melo2020iteration}. We also conduct numerical experiments to validate our claims. 


\subsection{Related Works}
Our algorithmic development is deeply rooted in the classical ALM  \cite{hestenes1969multiplier, powell1967method}, which was proposed in the late 1960s and is still one of the most powerful frameworks for solving constrained optimization problems. As the asymptotic convergence and convergence rate of ALM have been studied for convex programs \cite{rockafellar1973multiplier,rockafellar1976augmented} and smooth nonlinear programs \cite{bertsekas2014constrained}, we review some recent development of first-order algorithms based on ALM (including penalty methods) for problems of the form 
\begin{align}\label{eq: alm literature}
	\min_{x\in \R^n} \left\{ f(x) + g(x) ~|~ Ax = b,\  h(x) \leq 0\right\},
\end{align}
where usually $f$ is assumed to have a Lipschitz gradient, $g$ is a nonsmooth convex function whose proximal oracle is available, and constraints are convex. Note that though it is possible to write $Ax = b$ as $Ax\leq b$ and $-Ax \leq -b$, we distinguish affine equality constraints $Ax=b$ from nonlinear constraints $h(x) \leq 0$ as one usually needs to assume Slater-like conditions or linearly independence constraint qualification (LICQ) for the purpose of theoretical analysis. 

When only affine constraints are present and $g = 0$, Hong \cite{hong2016decomposing} proposed a proximal primal-dual algorithm (prox-PDA) that finds an  $\epsilon$-stationary point in $\mathcal{O}(\epsilon^{-2})$ iterations. Assuming $g$ has a compact domain, Hajinezhad and Hong \cite{hajinezhad2019perturbed} proposed a perturbed prox-PDA that achieves an iteration complexity of $\mathcal{O}(\epsilon^{-4})$. Zeng et al. \cite{zeng2021moreau} proposed a Moreau Envelope ALM (MEAL) with a iteration complexity of $\mathcal{O}(\epsilon^{-2})$ under some technical assumptions on a potential function. 

Note that the iteration complexities of above works are measured by the number of times a (proximal) augmented Lagrangian relaxation is solved. In contrast, there are works that report iteration complexities in terms of \textit{first-order oracles}, i.e., the number of proximal gradient steps, which is more tractable than directly solving the (proximal) augmented Lagrangian relaxation.
We summarize their results and compare to ours in Table \ref{table: summary}, where the term ``ccp" refers to $g$ being closed, convex, and proper.
Lin et al. \cite{lin2019inexact} proposed an inexact proximal point penalty (iPPP) method, where each subproblem is further solved by an adaptive accelerated proximal gradient (APG) steps with some specific line search strategies. The outer penalty method is later improved through a combination with ALM by Li and Xu \cite{LiQu2019alm}. Under a uniform Slater condition, both \cite{LiQu2019alm} and \cite{lin2019inexact} derived an $\tilde{\mathcal{O}}(\epsilon^{-2.5})$ first-order complexity bound. Melo et al. \cite{melo2020iteration, melo2020almfullstep} utilized an accelerated composite gradient method (ACG) \cite{beck2009fast} to solve each proximal ALM subproblem, showing that an $\epsilon$-stationary point can be found in $\tilde{\mathcal{O}}(\epsilon^{-3})$ ACG updates. Later, authors of \cite{melo2020iteration} embedded this inner acceleration scheme of \cite{melo2020iteration, melo2020almfullstep} in a proximal ALM framework to handle nonlinear constraints; they also obtained an improved rate of $\tilde{\mathcal{O}}(\epsilon^{-2.5})$ under a Slater-like condition. Work \cite{sun2022algorithms} considered the case where the objective is a difference-of-convex (DC) function, and combined a smoothing technique with ALM to solve affinely constrained DC programs.
\begin{table}[H]
\begin{center}
\begin{tabular}{c|c|c|c|c|c}
\toprule
Reference & $Ax=b$ & $h(x) \leq 0$ & $g$ & Loops  & First-order Complex. \\
\hline
\cite{melo2020almfullstep} &  \cmark     & \xmark   & ccp with bounded domain &   2  &   $\tilde{\mathcal{O}}(\epsilon^{-3})$    \\
\cite{kong2023iteration}   &  \cmark     &  \cmark   & ccp with bounded domain &   2  &   $\tilde{\mathcal{O}}(\epsilon^{-3})$    \\
\cite{sun2022algorithms}   &  \cmark     & \xmark    & difference-of-convex &   2  &   $\tilde{\mathcal{O}}(\epsilon^{-3})$    \\
\cite{LiXu2020almv2, lin2019inexact}   &  \cmark &  \cmark  & \xmark  &   3  &   $\tilde{\mathcal{O}}(\epsilon^{-2.5})$     \\ 
 \cite{zhang2020proximal, zhang2020global}  &   \cmark &  \xmark  & indicator of polyhedron  &  1   & $\mathcal{O}(\epsilon^{-2})$\\
 \cite{zhang2022iteration} &   \cmark &  \xmark  &  \eqref{eq: X} &  1   & $\mathcal{O}(\epsilon^{-2})$  \\
 This work &   \xmark &  \cmark  & \eqref{eq: X} & 1   & $\mathcal{O}(\epsilon^{-2})$   \\
\bottomrule
\end{tabular}
\caption{Summary of existing algorithms and the proposed method} \label{table: summary}
\end{center}
\end{table}

In contrast to all previously mentioned works,  Zhang and Luo proposed a single-looped proximal ALM and established an $\mathcal{O}(\epsilon^{-2})$ first-order (and also iteration) complexity when $g =\delta_X$ is an indicator function of a hypercube \cite{zhang2020proximal} or a polyhedron \cite{zhang2020global}. Their analysis is further generalized in \cite{zhang2022iteration} to handle the case where $X$ is a general convex set defined by smooth constraints. This paper is built upon this line of works, replacing explicit affine constraints $Ax=b$ by nonlinear constraints $h(x) \leq 0$. However, our modifications in algorithmic design and convergence analysis are not trivial. In fact, the framework of ALM is powerful enough to handle even non-convex constraints. Since we consider convex constraints in this paper, we do not delve extensively into these bodies of literature;  for example, see \cite{li2021rate, Sahin2019alm, sun2021dual} and references therein for some recent efforts. It would be interesting to generalize the proposed method to handle nonconvex constraints, and we reserve it for our future investigations. 

Other methods for constrained problems include those based on the proximal point method \cite{boob2022stochastic, boob2022level} or the sequential quadratic program (SQP)  \cite{curtis2021inexact, berahas2021sequential, berahas2022accelerating}. Though capable of handling more general constraints, in the context of problem \eqref{eq: nlp}, these methods ether lack explicit first-order complexity results or yield higher complexities when invoking first-order methods for their subproblems. 

\subsection{Our Contributions}
{Our method builds upon the foundation laid by Zhang et al. \cite{zhang2022iteration}. Unlike the projection requirements for complex convex sets in previous works, our approach assumes simpler constraints \( g(x) \leq 0 \) with a readily available projection oracle. We address more complex constraints \( h(x) \leq 0 \) through dualization, focusing on solving the Lagrangian rather than projecting onto difficult feasible regions. Our contributions are:
\begin{itemize}
    \item \textbf{$\mathcal{O}(1/\epsilon^2)$ Iteration Complexity:} We propose a single-looped first-order method for solving optimization problems in the form specified in \eqref{eq: nlp}. This method exhibits an iteration complexity of $\mathcal{O}(1/\epsilon^2)$, which improves upon previous bounds under comparable assumptions. Table \ref{table: summary} highlights these advancements, offering a new perspective on the efficiency of existing algorithms. Our approach simplifies implementation by eliminating the need for multiple nested loops and complex stopping conditions within these loops.
    \item \textbf{Containment Strategy for Dual Variable:} The introduction of non-linear constraints complicates the smoothness conditions typically associated with the Lagrangian, motivating the use of a bounding technique for dual variables. Specifically, we confine dual variables within a pre-selected hypercube \([0, B]^m\). This containment strategy, while common in practice, lacked a formal theoretical justification until now. We close this gap by establishing new error bounds that accommodate the artificial bounding of dual variables.
    \item \textbf{A New Dual Error Bound:} The core of our theoretical analysis lies in deriving error bounds for scenarios where dual variables are bounded. This aspect differs significantly from references \cite{zhang2020global, zhang2020proximal, zhang2022iteration}, where previous bounds were based on assumptions like the Hoffman bound associated with polyhedral sets \(X\). Our results extend the applicability of these bounds to more general settings involving explicit non-linear constraints.
\end{itemize}
}

\subsection{Notations and Organization}
We denote the set of positive integers up to $m$, the set of nonnegative integers, the set of real numbers, and the $n$-dimensional (nonnegative) real Euclidean space by $[m]$, $\N$, $\R$, and $\R^n_{(+)}$, respectively. 
The inner product of $x,y\in \R^n$ is denoted by $x^\top y$ or $\langle x, y \rangle$, and the Euclidean norm of $x$ is denoted by $\|x\|$. 
For a matrix $Q\in \R^{m\times n}$, we denote its smallest and largest eigenvalues by $\sigma_{\min}(Q)$ and $\|Q\|$; for an index set $\mathcal{S} \subseteq [m]$, we use $Q_{\mathcal{S}} \in \R^{|\mathcal{S}| \times n}$ to denote the submatrix of $A$ consisting of rows in $\mathcal{S}$.
For a differentiable mapping $h:\R^n \rightarrow \R^m$, we write $\nabla h(x) = [\nabla h_1(x), \cdots, \nabla h_m(x)]^\top \in \R^{m \times n}$. For a set $X\subseteq \R^n$, we use $\delta_X$ to denote its $0/\infty$-indicator function, and $\Pi_X(\cdot)$ to denote its projection operator in $\R^n$ with respect to the Euclidean norm. 

Since we need to deal with two sets of nonlinear constraints $h:\R^n \rightarrow \R^m$ and $g:\R^n \rightarrow \R^l$, we use $h_i :\R^n \rightarrow \R$ for $i \in [m]$ and $g_j:\R^n \rightarrow \R$ for $j \in [l]$ to denote individual components of $h$ and $g$, respectively. In case we need to consider $h$ and $g$ jointly, we index components of $g$ by $g_l$ for $l \in [m+1:m+l]: = \{m+1, \cdots, m+l\}$. Other notations will be explained upon their first appearance.

The rest of this paper is organized as follows. We provide some preliminaries in Section \ref{sec: preliminaries}, including definitions,  assumptions, and notations. In Section \ref{sec: convergence}, we present the convergence analysis of the proposed algorithm; for ease of reading, we put some technical proofs in the Appendix. The proposed algorithm is numerically validated in Section \ref{sec: experiments}, and we leave some final remarks as well as future directions in Section \ref{sec: conclusion}. 
\section{Preliminaries}\label{sec: preliminaries}
\subsection{Assumptions and Stationary Solutions}
We formally state assumptions on $(f, h, X)$ and define approximate stationarity for problem \eqref{eq: nlp}. 
\begin{assumption}[Basic assumptions]\label{assumption: basic assumption}
	We make the following basic assumptions.
	\begin{enumerate}
		\item The set $X$ is convex, compact, and described by \eqref{eq: X}. For each $j \in [l]$, the function $g_j:\R^n \rightarrow \R$ is convex with a Lipschitz gradient, i.e., there exists $L_{g_j} > 0$ such that for all $u,v\in \R^n$, $$\|\nabla g_j(u) - \nabla g_j(v)\| \leq L_{g_j}\|u-v\|. $$
		\item The function $f:\R^n \rightarrow \R$ is continuously differentiable over $X$, and there exists $L_f >0$ such that for all $u,v\in X$, $$\|\nabla f(u) - \nabla f(v)\|\leq L_f \|u-v\|.$$ 
		\item For each $i \in [m]$, the function $h_i: \R^n \rightarrow \R$ is convex and continuously differentiable over $X$, and there exists $L_{h_i}$ such that for all $u,v\in X$, $$\|\nabla h_i(u) - \nabla h_i(v) \| \leq L_{h_i}\|u - v\|. $$
	\end{enumerate}
\end{assumption}

By Assumption \ref{assumption: basic assumption}, the following constants are well-defined.  
\begin{subequations} \label{eq: problem constants}
\begin{align}
	& \nabla_f := \max_{x \in X} \|\nabla f(x)\|, \quad D_X := \max_{u, v \in X} \|u-v\|, \quad \underline{f} :=  \min_{x\in X} f(x), \\
	& L_h : = \sqrt{\sum_{i \in [m]} L_{h_i}^2},  \quad L_g : = \sqrt{\sum_{j \in [l]} L_{g_j}^2}, \quad M_h: = \max_{x\in X}\|h(x)\|, \\ 
    & K_{h_i} := \max_{x\in X} \|\nabla h_i(x)\|, ~\forall i \in [m], \quad K_h : = \sqrt{\sum_{i \in [m]} K_{h_i}^2}. 
\end{align}
\end{subequations}
Without loss of generality, we assume that the above constants are strictly positive; otherwise, the problem can be reduced to a much simpler one. It is easy to verify that for all $u,v \in X$, the following inequalities hold:
\begin{align}
	 |h_i(u) - h_i(v)| \leq & K_{h_i}\|u - v\|, ~\forall i \in [m], \\
	 \|h(u) - h(v)\| \leq & K_h \|u - v\|, \\
	 \|\nabla h(u) - \nabla  h(v) \| \leq & L_h \|u - v\|.
\end{align}

Our proposed method seeks for approximate stationary solutions of problem \eqref{eq: nlp} in the following sense.
\begin{definition}[Stationary point to nonlinear program \eqref{eq: nlp}]\label{def: stationary}
	Given $\epsilon \geq 0$, a vector $x \in X$ is an $\epsilon$-\textit{stationary point} of problem \eqref{eq: nlp} if there exist $(\xi, y) \in \R^n \times \R^m_+$ such that 
	\begin{subequations}
	\begin{align}
		& \xi \in \nabla f(x) + \nabla h(x) ^\top y + N_{X}(x),~ \|\xi\|\leq \epsilon, \label{eq: nlp stationary dual infeas}\\
		& \|\Pi_{\R^m_+}(h(x))\| \leq \epsilon,  \label{eq: nlp stationary primal infeas} \\
        & |\langle y,  h(x)\rangle| \leq \epsilon,  \label{eq: nlp stationary cs}
 	\end{align}
    \end{subequations}
	where $N_X(x)$ denotes the normal cone of $X$ at $x$. We simply call $x$ a stationary point when $\epsilon = 0$, and we denote the set of all stationary points by $X^*$.
\end{definition}
\begin{remark}
    Conditions \eqref{eq: nlp stationary dual infeas}, \eqref{eq: nlp stationary primal infeas}, and \eqref{eq: nlp stationary cs} correspond to natural relaxations of dual feasibility, primal feasibility, and complimentary slackness, respectively.
\end{remark}

In addition to Assumption \ref{assumption: basic assumption}, our analysis relies on certain regularity conditions as follows.  
\begin{definition}[Active set and Jacobian matrix]\label{def: active set}
	In view of problem \eqref{eq: nlp}, the active set of some $x \in X$ is defined as
	\begin{align*}
		A[x]: = \{i \in [m]:~ h_i(x) = 0 \} \cup \{j \in [m+1:m+l]:~g_j(x) = 0\},
	\end{align*}
	which is a subset of $[m+l]$. Further define the Jacobian of all constraints of \eqref{eq: nlp} as 
	\begin{align*}
		J(x):= [\nabla h_1(x), \cdots, \nabla h_m(x), \nabla g_{m+1}(x), \cdots, \nabla g_{m+l}(x)]^\top \in \R^{(m+l) \times n}
	\end{align*}
	We use $J_{A}(x)\in \R^{|A[x]| \times n}$ to denote the submatrix of $J(x)$ consisting of rows in $A[x]$.
\end{definition}

\begin{assumption}[Regularity]\label{assumption: convex regularity}
	We make the following regularity assumptions on the problem. 
	\begin{enumerate}
		\item There exist $\hat{x} \in X$ and $\Delta_0 > 0$ such that $h_i(\hat{x} ) \leq  -\Delta_0$ for all $i \in [m]$.
		\item For any $x^* \in X^*$ and any $\mathcal{S}_A \subseteq A[x^*]$, the smallest singular value of $J_{\mathcal{S}_A}(x^*)^\top$ is bounded away from zero, where $J_{\mathcal{S}_A}(x^*)$ denotes a submatrix of $J_A(x^*)$ consisting of rows specified in $\mathcal{S}$. Equivalently, we can define the following positive constant:
        \begin{align*}
            \sigma_{X^*} := \inf_{ \stackrel{x^* \in X^*}{\mathcal{S}_A \subseteq A[x^*]} } \sigma_{\min}(J_{\mathcal{S}_A}(x^*)^\top)  > 0. 
        \end{align*}
	\end{enumerate}
\end{assumption}
\begin{remark}
    We give some remarks regarding Assumption \ref{assumption: convex regularity}.
    \begin{enumerate}
        \item Notice that part 1 is a Slater-like commonly adopted for the analysis of convex programs. However, our assumption is slightly weaker in the sense that we do not require $\hat{x}$ to be strictly feasible with respect to $g(x) \leq 0$, i.e., $\hat{x}$ does not have to belong to the (relative) interior of $X$. 
        \item Part 2 requires all stationary point of problem \eqref{eq: nlp} should satisfy the LICQ, which is also a common assumption for the nonlinear programs. As we do not have prior knowledge on which stationary point $x^*$ the proposed method will converge to, we further assume a uniform lower bound $\sigma_{X^*}$ over the set of all stationary points $X^*$ in order to provide some qualitative control of $J_{S_A}(x^*)$. Intuitively speaking, we want to make sure that all rows in $J_{S_A}(x^*)$ are ``sufficiently linearly independent". In addition, we emphasize that this condition is only assumed on stationary points, rather than all feasible points, of \eqref{eq: nlp} 
    \end{enumerate}
\end{remark}

\subsection{Useful Notations}
We introduce some notations that will be helpful in the analysis. Recall that $(\alpha, \beta, c, p)$ are algorithmic parameters and $B$ is an upper bound of each component of $Y$. Their exact ranges will be specified later. We will frequently adopt the following notations in our analysis. 
\begin{itemize}
\item For some $y \in \R^m_+$, define convex functions
\begin{align}
    h_y(x): = \langle y, h(x)\rangle, \text{~and~} h_Y(x) := \max_{y \in Y}h_y(x).
\end{align}
For all $y\in Y$, the function $h_y(x)$ has a Lipschitz gradient with modulus bounded by

\begin{align}
    L_h(Y) : = L_h \max_{y\in Y}\|y\| = \sqrt{m}L_hB.
\end{align}
\item Choose $p > L_f$. For any $(z, y) \in X \times Y$, the function $K(\cdot ,z;y)$ is strongly convex with modulus 
\begin{align} \label{eq: p range}
	 \mu_K := p -L_f > 0. 
\end{align}
In addition, the following quantities are well-defined:
\begin{align}
    x(z) := & \arg\min_{x\in X} f(x) + h_Y(x) + \frac{p}{2}\|x - z\|^2 = \argmin_{x\in X} \big(\max_{y\in Y} K(x,z; y)\big), \\
    P(z): = & \min_{x\in X} f(x) + h_Y(x) + \frac{p}{2}\|x - z\|^2 = f(x(z)) + h_Y(x(z)) + \frac{p}{2}\|x(z) - z\|^2, \\
    x(y, z) := & \arg\min_{x\in X}K(x, z; y),  \\
    d(y, z) := & \min_{x \in X} K(x, z; y) = K(x(y,z), z; y), \\
    \nabla_y d(y, z) = &  \nabla_y K(x(y,z), z; y) = h(x(y,z)),\\
    Y(z) := & \Argmax_{y\in Y} d(y, z).
\end{align}
\item Finally, for some given $\alpha > 0 $ and $c> 0$,  define
\begin{align}
	y_+(z) := &\Pi_Y(y+\alpha h(x(y, z))),  \label{eq: y_+(z)}\\
	x_+(y, z) := &\Pi_X(x-c\nabla_xK(x, z; y)). \label{eq: x_+(y, z)}
\end{align}
These quantities are closely related to dual and primal iterates generated by the proposed algorithm, and their differences can be controlled by certain error bounds. Keep in mind that $y_+(z)$ and $x_+(y, z)$ are parameterized by $\alpha$ and $c$, respectively; we skip such dependencies in the above notations for ease of presentation. 
\end{itemize}
The following facts will be useful in our analysis. 
\begin{lemma}\label{lemma: minimax thm}
   For any $z \in \R^n$, it holds that 
    \begin{align}
        & P(z) = \min_{x\in X} \max_{y \in Y} K(x,z;y) = \max_{y \in Y} \min_{x\in X} K(x,z;y) = \max_{y \in Y} d(y, z), \label{eq: minimax1}\\
        & x(y(z), z) = x(z) \text{~for any~} y(z) \in Y(z). \label{eq: minimax2}
    \end{align}
\end{lemma}
\begin{proof}
    The second equality in \eqref{eq: minimax1} is due to the Minimax Theorem, e.g., see \cite[Theorem 3.1.29]{nesterov2018lectures}. The second claim \eqref{eq: minimax2} is proved in  \cite[Lemma B.8]{zhang2020single}. 
\end{proof}

\section{Convergence Analysis}\label{sec: convergence}

\subsection{A Potential Function and its Basic Descent Property}
For $t \in \N$, define 
\begin{align}
	\Phi^t = \Phi(x^t, z^t;y^t):= K(x^t, z^t;y^t)-2d(y^t, z^t)+2P(z^t).
\end{align}
The core of our analysis is to show that $\Phi$ can serve as a potential function, which generates a non-increasing sequence, stays bounded from below, and admits sufficient descents over iterations. The next lemma establishes a basic estimate of the difference between $\Phi^t$ and $\Phi^{t+1}$. 

\begin{lemma}\label{lemma: potential function}
    Suppose Assumption \ref{assumption: basic assumption} holds. Define constants 
	\begin{align}\label{eq: kappa123-lemma}
		\kappa_1 := \frac{p}{\mu_K}, ~ \kappa_2 := \frac{K_h}{\mu_K}, \text{~and~} \kappa_3:=	1 + \frac{1}{c \mu_K}. 
	\end{align}
    Suppose we select $p$ so that  \eqref{eq: p range} holds and $(c, \alpha, \beta)$ (in this order) such that
    \begin{align} 
    	 & 0 < c \leq \min \left\{\frac{1}{4K_h}, \frac{\kappa_1}{16 p \kappa_2^2}, \frac{1}{L_f + L_h(Y) + p} \right\}, \label{eq: c_range}\\
        & 0 < \alpha \leq \min \left\{ \frac{3}{4(K_h\kappa_3^2 -K_h \kappa_2 )},  \frac{1}{4cK_h^2 \kappa_3^2}\right\}, \label{eq: alpha_range} 
     \end{align}
    and $0 < \beta \leq \frac{1}{24\kappa_1}$. Then for all $t \in \N$, it holds that 
    \begin{align}
    \Phi^t - \Phi^{t+1} \geq & \frac{1}{8c} \|x^t - x^{t+1}\|^2 + \frac{1}{ 8 \alpha} \|y^t - y_+^t(z^t)\|^2  +\frac{p}{8\beta} \|z^t - z^{t+1}\|^2 - 24 p \beta \|x(z^t)-x(y^t_+(z^t), z^t)\|^2. \label{eq: basic descent inequality}
\end{align}
\end{lemma}
The proof of Lemma \ref{lemma: potential function} can be found in Appendix \ref{appendix: potential function}. 

\subsection{Dual Error Bounds}
Note that Lemma \ref{lemma: potential function} does not show the sequence $\{\Phi^t\}_{t\in \N}$ is non-increasing, as there is a negative term $- 24 p \beta \|x(z^t)-x(y^t_+(z^t), z^t)\|^2$ in the right-hand side of \eqref{eq: basic descent inequality}. Our strategy, roughly speaking, is to show that the three positive terms $ \frac{1}{8c} \|x^t - x^{t+1}\|^2 + \frac{1}{ 8 \alpha} \|y^t - y_+^t(z^t)\|^2  +\frac{p}{8\beta} \|z^t - z^{t+1}\|^2$ is able to compensate for the negative term. To this end, we state two \textit{dual error bounds} in this subsection. We first adopt a weak dual error bound from \cite{zhang2020single}. 
\begin{lemma} [Lemma D.1 of \cite{zhang2020single}]\label{lemma: weak eb} 
	Suppose Assumption \ref{assumption: basic assumption} holds. Define 
	\begin{align}\label{eq: kappa4}
		\kappa_4:= \frac{\sqrt{m} B (1 +  \alpha (L_f + L_h(Y))(1+\kappa_2))}{\alpha \mu_K}.
	\end{align}
	Then for all $z\in X$ and $y \in Y$, it holds that 
	\begin{align}
		\|x(z) - x(y_+(z),z)\|^2 \leq \kappa_4 \| y - y_+(z)\|. \notag 
	\end{align}
\end{lemma}
The above error bound is ``weak" as it is inhomogeneous. In the next proposition, we show that if we choose $B$ to be sufficiently large, then the following ``strong" dual error bound holds. 

\begin{proposition}\label{prop: strong eb}
Suppose Assumptions \ref{assumption: basic assumption} and \ref{assumption: convex regularity} hold, and choose $B$ such that 
\begin{align}\label{eq: b_rule0}
        B >  \max \left \{ \frac{4\nabla_f D_X + 2 pD_X^2}{\Delta_0},  \frac{2\nabla_f + 2pD_X}{\sigma_{X^*}} \right\}. 
\end{align}
There exists constant $\delta > 0$ such that if $\max\{\|y^t-y_+^t(z^t)\|, \|z^t-x(z^t)\|\}\le \delta$, then we have
\begin{align} \notag 
    \|x(y^{t+1},z^t)-x(z^t)\|\le \kappa_5\|y^t-y_+^t(z^t)\|
\end{align}
 where 
 \begin{align} \label{eq: kappa_5}
    \kappa_5: =  \frac{2\left(L_f + p + (2\nabla_f + 2pD_X) \sigma_{X^*}^{-1}\sqrt{L_h^2 + L_g^2} \right) \left(\frac{1}{\alpha} + K_h \kappa_2 \right) }{\mu_K \sigma_{X^*}}. 
 \end{align}
\end{proposition}
The proof of Proposition \ref{prop: strong eb} is rather technical. We provide details in Appendix \ref{sec: proof of strong eb}.

\subsection{Convergence and Iteration Complexity}
Combining Lemma \ref{lemma: potential function} and the dual error bounds from the previous subsection, we can show that the sequence $\{\Phi^t\}_{t \in \N}$ is indeed non-increasing.
\begin{lemma}\label{lemma: potential function descent}
    Suppose Assumptions \ref{assumption: basic assumption} and \ref{assumption: convex regularity} hold. Recall $(\kappa_1, \kappa_2, \kappa_3, \kappa_4, \kappa_5)$ from \eqref{eq: kappa123-lemma}, \eqref{eq: kappa4}, and \eqref{eq: kappa_5}, and $\delta > 0$ specified in Proposition \ref{prop: strong eb}. Suppose we choose  $c$ according to \eqref{eq: c_range},  $\alpha$ according to \eqref{eq: alpha_range}, $\beta$ such that 
    \begin{align}\label{eq: beta_range2}
        0 < \beta \leq \min \Bigg\{ & 
        \frac{1}{24\kappa_1},  
        \frac{1}{384\alpha p \kappa_5^2}, 
        \frac{\delta^2}{384D_X^2 p \alpha}, 
        \frac{\delta^2}{6144\kappa_2^2 D_X^2 p \alpha} , \notag \\
        & \frac{\delta^2 }{6144\kappa_3^2 D_X^2pc},  
        \frac{\delta^4}{(3080\kappa_4)^2 \times 384 D_X^2 p\alpha} \Bigg\}.
    \end{align}
    and $B$ according to \eqref{eq: b_rule0}. Then for $t \in \N$, we have 
    \begin{align}\label{eq: potential function descent}
    	\Phi^t-\Phi^{t+1}&\ge\frac{1}{16c}\|x^t-x^{t+1}\|^2+\frac{1}{16\alpha}\|y^t-y^t_+(z^t)\|^2+\frac{p}{16\beta}\|z^t-z^{t+1}\|^2. 
    \end{align}
\end{lemma}
\begin{proof}
 Consider the following three conditions
    \begin{align}
        \|x^t - x^{t+1}\|^2 \leq &   384D_X^2 p c \beta, \label{eq: condition1}\\
        \|y^t - y_+^t(z^t)\|^2 \leq &  384D_X^2 p \alpha \beta, \label{eq: condition2}\\
        \|z^t - x^{t+1}\|^2 \leq & 384 \|x(z^t) - x(y^t_+(z^t), z^t)\|^2. \label{eq: condition3}
    \end{align}
    We consider two cases next. 
    \begin{enumerate}
        \item Conditions \eqref{eq: condition1}-\eqref{eq: condition3} hold. Due to the choice of $\beta$ in \eqref{eq: beta_range2} and the fact that $\kappa_3 >1$, we have 
        \begin{align}
            & \|x^t - x^{t+1}\| \leq  \sqrt{384D_X^2 pc \beta} \leq \sqrt{\frac{384 D_X^2 pc \delta^2}{ 6144 \kappa_3^2 D_X^2 pc} } \leq \delta, \notag \\
            & \|y^t - y_+^t(z^t)\| \leq \sqrt{384D_X^2 p \alpha \beta} \leq \delta. \notag
        \end{align}
        In addition, it holds that
        \begin{align*}
            & \|z^t - x(z^t)\|^2 \\
            \leq  & 4 \|z^t - x^{t+1}\|^2 + 4\|x^{t+1} - x(y^t, z^t) \|^2 \\
            &+4\|x(y^t,z^t) - x(y^t_+(z^t),z^t)\|^2 + 4\|x(y^t_+(z^t), z^t) - x(z^t)\|^2 \\
            \leq & 1540 \|x(y^t_+(z^t), z^t) - x(z^t)\|^2 + 4 \kappa_3^2 \|x^t - x^{t+1}\|^2 + 4 \kappa_2^2 \|y^t - y^t_+(z^t)\|^2 \\
            \leq & 1540 \kappa_4 \sqrt{384D_X^2 p \alpha \beta } + 1536 \kappa_3^2 D_X^2 p c \beta + 1536 \kappa_2^2 D_X^2 p \alpha \beta \\
            \leq & \frac{\delta^2}{2} + \frac{\delta^2}{4} + \frac{\delta^2}{4} = \delta^2, 
        \end{align*}
        where the second inequality is due to \eqref{eq: condition3}, \eqref{eb2}, \eqref{eb3}, the third inequality is due to Lemma \ref{lemma: weak eb}, \eqref{eq: condition1} and \eqref{eq: condition2}, and the last inequality is due to the choice of $\beta$ in \eqref{eq: beta_range2}. Now we have 
        \begin{align*}
			&\frac{1}{8\alpha}\|y^t-y_+^t(z^t)\|^2-24p\beta\|x(z^{t})-x(y^t_+(z^t), z^{t})\|^2\nonumber\\
			= & \frac{1}{16\alpha}\|y^t-y_+^t(z^t)\|^2+\frac{1}{16\alpha}\|y^t-y_+^t(z^t)\|^2-24p\beta\|x(z^{t})-x(y^t_+(z^t), z^{t})\|^2\nonumber\\
			\geq & \frac{1}{16\alpha}\|y^t-y_+^t(z^t)\|^2+\frac{1}{16\alpha}\|y^t-y_+^t(z^t)\|^2-24p\beta\kappa_5^2\|y^t-y_+^t(z^t)\|^2\nonumber\\
			\geq &\frac{1}{16\alpha}\|y^t-y_+^t(z^t)\|^2,
		\end{align*}        
       	where the first inequality is due to the dual error bound established in Proposition \ref{prop: strong eb}, and the second inequality is due to the choice of $\beta$ in \eqref{eq: beta_range2}.

        \item One of \eqref{eq: condition1}-\eqref{eq: condition3} is violated. We consider each of the three possibilities. 
        \begin{enumerate}
        	\item Condition \eqref{eq: condition1} is violated. Suppose $\|x^t - x^{t+1}\|^2  > 384D_X^2 p c \beta$, then
        	\begin{align*}
				&\frac{1}{8c}\|x^t-x^{t+1}\|^2-24p\beta\|x(z^{t})-x(y^t_+(z^t), z^{t})\|^2\nonumber\\
				\geq &\frac{1}{16c}\|x^t-x^{t+1}\|^2+\frac{1}{16c}384D_X^2pc\beta-24 D_X^2 p \beta=\frac{1}{16c}\|x^t-x^{t+1}\|^2.
			\end{align*}
        	\item Condition \eqref{eq: condition2} is violated. Suppose $ \|y^t - y_+^t(z^t)\|^2 > 384D_X^2 p \alpha \beta$, then 
        	\begin{align*}
        		&\frac{1}{8\alpha}\|y^t-y_+^t(z^t)\|^2-24p\beta\|x(z^{t})-x(y^t_+(z^t), z^{t})\|^2\nonumber\\
				\geq & \frac{1}{16\alpha}\|y^t-y_+^t(z^t)\|^2+\frac{1}{16\alpha}384 D_X^2p \alpha \beta-24 D_X^2 p \beta=\frac{1}{16\alpha}\|y^t-y_+^t(z^t)\|^2.
        	\end{align*}
        	\item Condition \eqref{eq: condition3} is violated. Suppose $\|z^t - x^{t+1}\|^2 > 384 \|x(z^t) - x(y^t_+(z^t), z^t)\|^2$, then 
        	\begin{align*}
        			 & \frac{p}{8\beta}\|z^t-z^{t+1}\|^2-24p\beta\|x(z^{t})-x(y^t_+(z^t), z^{t})\|^2\nonumber\\
				= 	 & \frac{p}{16\beta}\|z^t-z^{t+1}\|^2+\frac{p\beta}{16}\|z^t-x^{t+1}\|^2-24p\beta\|x(z^{t})-x(y^t_+(z^t), z^{t})\|^2\\
				\geq & \frac{p}{16\beta}\|z^t-z^{t+1}\|^2,
        	\end{align*}
        	where the equality is due to the update $z^{k+1} = z^k + \beta(x^{k+1}-z^k)$. 
        \end{enumerate}
    \end{enumerate}
    In both cases, the claimed inequality holds. This completes the proof. 
\end{proof}

The next proposition shows that if we choose $B$ large enough, and if $\|x(y,z)-z\|$ and $\|y_+(z) - y\|$ are small, then $x(y,z)$ are almost feasible for problem \eqref{eq: nlp}. 
\begin{lemma}\label{lemma: primal feas and cs}
	Suppose Assumption \ref{assumption: basic assumption} holds, and we choose $B$ such that 
 \begin{align}\label{eq: B rule}
    B >  \underline{B} := \max \left\{ \frac{4\nabla_f D_X + 4m K_h D_X + 2pD_X^2}{\Delta_0} + 1, \frac{2\nabla_f + 2pD_X}{\sigma_{X^*}} \right\}. 
\end{align}
	If for some $z\in X$ and $y \in Y$, we have $\max\{ \|x(y,z) - z\|, \|y_+(z) - y\|\} \leq \eta$ for some $ 0 < \eta \leq \overline{\eta}$, where 
	\begin{align}
		\overline{\eta} : = \min \left \{ \frac{1}{4} \Delta_0 \alpha, 1 \right \}
	\end{align}
    Then it holds that 
    \begin{align}
		\|\Pi_{\R^m_+}(h(x(y,z)))\|  \leq \frac{ \sqrt{m} \eta}{\alpha}, \quad |\langle y_+(z), h(x(y,z))| \leq \frac{m\underline{B} \eta}{\alpha}. 
	\end{align}
\end{lemma}
\begin{proof}
	For each $i \in [m]$, denote the $i$-th component of $y_+(z)$ by $y^+_i$. 
    Suppose $h_i$'s are convex and Assumption \ref{assumption: convex regularity} holds.   Consider the condition $\|y_+(z) - y\| \leq \delta$. Each $i \in [m]$ must fall into one of the following three sets. 
    \begin{enumerate}
        \item $I_1 = \{i \in [m]:~y^+_i = 0\}$. Since $|y^+_i - y_i| \leq \|y_+(z) - y\| \leq \eta$, we must have $y_j\in [0, \eta]$ and $h_i(x(y,z))\leq 0$. 
        \item $I_2 = \{i \in [m]:~y^+_i \in (0, \underline{B}]\}$. Since $y^+_i$ belongs to the interior of $[0,B]$, we must have $y^+_i = y_i + \alpha h_i(x(y,z))$, and
        \begin{align}\label{eq: bound on h_i violation}
            |h_i(x(y,z))| = \frac{1}{\alpha} |y^+_i - y_i| \leq  \frac{1}{\alpha} \|y_+(z) - y\| \leq \frac{\eta}{\alpha}. 
        \end{align}
        \item $I_3 = \{i \in [m]:~ y^+_i \in (\underline{B}, B]\}$. Notice that we have $y_i \geq y_i^+ - \eta \geq \underline{B} - 1$. In addition, there are two cases: either $y^+_j \leq B$ and $y^+_j = y_i + \alpha h_i(x(y,z))$, or $y^+_i = B$ and $ y_i + \alpha h_i(x(y,z)) > B$. 
        In the first case, \eqref{eq: bound on h_i violation} holds as well; in the second case, we have $h_i(x(y,z)) > 0$. 
    \end{enumerate}
    We claim that  $I_3 = \emptyset$ under Assumption \ref{assumption: convex regularity}. For the purpose of contradiction, suppose there exists at least one index $i \in I_3$. Recall that $\hat{x} \in X$ is a Slater point such that $h_i(\hat{x}) < -\Delta_0 $ for all $i \in [m]$. Then there exists $\theta \in (0, 1)$ such that the point defined by 
    \begin{align}
        x_\theta := \theta \hat{x} + (1-\theta) x(y,z) \in X
    \end{align}
    satisfies that
    \begin{align}
        h_i(x_\theta) \leq -\frac{1}{2} \Delta_0, ~ \forall i \in [m]. 
    \end{align}
    We show that $x_\theta$ achieves a lower objective than $x(y,z)$ in $\min_{x\in X} K(x,z;y)$, hence contradicting to the optimality of $x(y,z)$. For $j \in J_1$, we have 
    \begin{align}\label{eq: J1}
        y_i(h_i(x_\theta) - h_i(x(y,z))) \leq \eta K_h \|x_\theta - x(y,z)\| \leq \eta K_h D_X
    \end{align}
    For $i \in I_2$, since $\eta/\alpha \leq \frac{1}{2}\Delta_0$, we have $h_i(x_\theta) \leq -\frac{\Delta_0}{2} \leq -\eta/\alpha \leq h_i(x(y,z))$ and hence
     \begin{align}\label{eq: J2}
        y_i(h_i(x_\theta) - h_i(x(y,z))) \leq  0.
    \end{align}
    Finally for $i \in I_3$, since $y_i \geq \underline{B}-1$, and either $h_i(x(y,z)) \geq - \frac{\eta}{\alpha} \geq -\frac{\Delta_0}{4}$ or $h_i(x(y,z)) > 0$, we have
    \begin{align}\label{eq: J3}
        y_i(h_i(x_\theta) - h_i(x(y,z))) \leq - \frac{(\underline{B}-1)\Delta_0}{4}.
    \end{align}
    As a result, we have
    \begin{align*}
        & \left(f(x_\theta) + \langle y, h(x_\theta)) \rangle + \frac{p}{2}\|x_\theta - z\|^2\right) - \left(  f(x(y,z)) + \langle y, h(x(y,z)) \rangle + \frac{p}{2}\|x(y,z)- z\|^2 \right) \\
        \leq  & f(x_\theta) -  f(x(y,z)) + \sum_{j \in [m]} y_j (h_i(x_\theta) - h_i(x(y,z))) + \frac{p}{2}\|x_\theta - z\|^2 \\
        \leq & \nabla_f \|x_\theta - x(y,z)\| + \sum_{j \in J_1} y_j (h_i(x_\theta) - h_i(x(y,z))) + \sum_{j \in J_3} y_j (h_i(x_\theta) - h_i(x(y,z))) + \frac{p}{2}\|x_\theta - z\|^2 \\
        \leq & \nabla_f D_X + \frac{p}{2}D_X^2 + (m-1) \eta K_h D_X  -\frac{(\underline{B}-1)\Delta_0}{4} < 0
    \end{align*}
    where the second inequality is due to the Lipschitzness of $f$ over $X$ and \eqref{eq: J2}, and the last inequality is due to \eqref{eq: J1}, \eqref{eq: J3}, the compactness of $X$, as well as the claimed value of $\underline{B}$. This is a desired contradiction. As a result, we have either
    \begin{align}
        h_i(x(y,z)) \leq 0, ~  y^+_i = 0, \notag 
    \end{align}
    or 
    \begin{align} 
        |h_i(x(y,z))| \leq \frac{\eta}{ \alpha}, ~ y^+_i \in [0, \underline{B}]. \notag 
    \end{align}
	Consequently, it holds that 
	\begin{align*}
		\|\Pi_{R^m_+}(h(x(y,z)))\|  = & \sqrt{\sum_{i=1}^m \max\{0, h_i(x(y,z))\}^2} \leq \frac{ \sqrt{m} \eta}{\alpha}, \\
		|\langle y_+(z), h(x(y,z)) \rangle| \leq & \sum_{i=1}^m |y^+_i | |h_i(x(y,z))|  \leq \frac{m\underline{B} \eta}{\alpha}. 
	\end{align*}
	This completes the proof. 
\end{proof}

Now we are ready to present the iteration complexity of Algorithm \ref{Alg:SProxALM}

\begin{theorem}
	Suppose Assumptions \ref{assumption: basic assumption} and \ref{assumption: convex regularity} holds and select algorithmic parameters $(p, c, \alpha, \beta, B)$ according to \eqref{eq: p range}, \eqref{eq: c_range}, \eqref{eq: alpha_range}, \eqref{eq: beta_range2}, and \eqref{eq: B rule} respectively. Further define constants:
\begin{subequations}\label{eq: lambda}
\begin{align}
	\lambda_0 := &  \min\{1/(16c), 1/(16\alpha), p/(16\beta)\},  \\
	\lambda_1 := & L_f + \sqrt{m}B L_h + \frac{1}{c} + p + \frac{p}{\beta}, \\
	\lambda_2 := & \frac{\sqrt{m}(\kappa_3 + 1/\beta)}{\alpha} + K_h \kappa_3,  \\
	\lambda_3 := & M_h + \sqrt{m}\underline{B}K_h\kappa_3 + \frac{m\underline{B}(\kappa_3 + 1/\beta)}{\alpha}. 
\end{align}
Then given any $\epsilon >0$, Algorithm \ref{Alg:SProxALM} finds an $\epsilon$-stationary point of problem \eqref{eq: nlp} in no more than $T$ iterations, where 
\end{subequations}
\begin{align}\label{eq: T_ub}
	T \leq \left \lceil  \frac{ (\Phi^0 - \underline{f}) \max\{\lambda_1, \lambda_2, \lambda_3\}^2 }{\epsilon^2 \lambda_0 }\right \rceil = \mathcal{O}(\epsilon^2). 
\end{align}
\end{theorem}
\begin{proof}
	By Lemma \ref{lemma: potential function descent}, summing \eqref{eq: potential function descent} from $0$ to some positive index $T-1$, we have
	\begin{align}
		\sum_{t = 0}^{T-1} \left(  \frac{1}{16c}\|x^t - x^{t+1}\|^2 +  \frac{1}{16 \alpha} \|y^t - y^{t}_+(z^t)\|^2 + \frac{p}{16 \beta}  \|z^t -z^{t+1}\|\right) \leq \sum_{t = 0}^{T-1} \Phi^t - \Phi^{t+1} \leq \Phi^0 - \underline{f}, \notag 
	\end{align}
	where the last inequality is due to Lemma \ref{lemma: lower bounded}. As a result, there exists a specific $t \in \{0, 1, \cdots, T-1\}$ such that 
	\begin{align}
		\|x^t - x^{t+1}\|^2 + \|y^t - y^{t}_+(z^t)\|^2 + \|z^t -z^{t+1}\|^2  \leq \frac{\Phi^0 - \underline{f}}{ T \lambda_0}, \notag
	\end{align}
	which further implies that 
	\begin{align} \label{eq: bound max diff}
		\varrho_{t+1}:= \max\{\|x^t - x^{t+1}\|, \|y^t - y^{t}_+(z^t)\|,  \|z^t -z^{t+1}\| \} \leq \sqrt{\frac{\Phi^0 - \underline{f}}{T \lambda_0}}. 
	\end{align}
	Next we investigate the approximate stationarity of $(x^{t+1}, y^t)$. By the update of $x^{k+1}$, we have 
	\begin{align}
		 \xi^t : =  &  \nabla f(x^{t+1}) - \nabla f(x^t) + \langle \nabla h(x^{t+1}) - \nabla h(x^t), y^t \rangle - p(x^t-z^t) - \frac{1}{c}(x^{t+1}-x^t) \notag  \\
		\in & \nabla f(x^{t+1}) +\nabla h(x^{t+1})^\top y^t + N_X(x^{t+1})	\notag,
	\end{align}
	and hence 
	\begin{align}
		\|\xi^t\| \leq & \|\nabla f(x^{t+1}) - \nabla f(x^t)\|	 + \|y^t\|\|\nabla h(x^{t+1}) - \nabla h(x^t) \| + p\|x^t-z^t\|+ \frac{1}{c}\|x^{t+1}-x^t\| \notag \\
		\leq & L_f\|x^t - x^{t+1}\| + \sqrt{m} BL_h\|x^t - x^{t+1}\| + p\|x^t-x^{t+1}\| + \frac{p}{\beta}\|z^t-z^{t+1}\| +  \frac{1}{c}\|x^{t+1}-x^t\|\notag  \\
		\leq & \left(L_f + \sqrt{m}B L_h + \frac{1}{c} + p + \frac{p}{\beta}\right) \varrho_t  = \lambda_1 \varrho_{t+1} \label{eq: d_error}
	\end{align}
	where the second inequality is due to $\nabla f$ being $L_f$-Lipschitz, $\nabla h$ being $L_h$ Lipschitz, $y\in Y$, and $\|x^t - z^t\| \leq \|x^t - x^{t+1}\| + \|x^{t+1}-z^t\| =\|x^t - x^{t+1}\|  + \frac{1}{\beta}\|z^t-z^{t+1}\|$. Also notice that by \eqref{eb3} and the update of $z^{k+1}$, we haves 
	\begin{align}
		\|x(y^t, z^t) - z^t\| \leq \|x(y^t, z^t) - x^{t+1}\| + \|x^{t+1} - z^t\| \leq \sigma_3\|x^t - x^{t+1}\| + \frac{1}{\beta}\|z^t- z^{t+1}\|, \notag
	\end{align}
	and therefore we have 
	\begin{align}
		& \max \{\|x(y^t, z^t) - z^t\|, \|y^t - y_+^t(z^t)\|\} \leq \eta_t := (\kappa_3 + 1/\beta)\varrho_{t+1}, \notag 
	\end{align}
	Further choose $T \geq \frac{\Phi^0 - \underline{f}}{\lambda(\alpha) \overline{\eta}^2}$ so that Lemma \ref{lemma: primal feas and cs} can be invoked. Next we consider primal infeasibility at $x^{t+1}$: 
	\begin{align}
		\|\Pi_{X}(h(x^{t+1}))\| \leq & \|\Pi_{X}(h(x(y^t, z^t))\| + \|\Pi_{X}(h(x(y^t, z^t)) - \Pi_{X}(h(x^{t+1}))\| \notag \\
		\leq & \frac{\sqrt{m}}{\alpha} \eta_t + \|h(x(y^t, z^t) - h(x^{t+1})\| \notag \\
		\leq & \frac{\sqrt{m}}{\alpha} \eta_t + K_h \kappa_3 \|x^t-x^{t+1}\|  \notag \\
		\leq & \left(\frac{\sqrt{m}(\kappa_3 + 1/\beta)}{\alpha} + K_h \kappa_3\right) \varrho_t = \lambda_2 \varrho_{t+1},\label{eq: p_error}
	\end{align}
	where the second inequality is due to the nonexpansiveness of the projection operator, the third inequality is due to $h$ being $K_h$-Lipschitz and \eqref{eb3}, and the last inequality is due to Lemma \ref{lemma: primal feas and cs}. Moreover, the violation of complementary slackness can be bounded by 
	\begin{align}
			& |\langle y^t, h(x^{t+1}) \rangle| \notag  \\
			\leq  & |\langle y^t - y_+^t(z^t), h(x^{t+1})\rangle| + |\langle y_+^t(z^t), h(x^{t+1}) - h(x(z^t, y^t))\rangle| + |\langle y_+^t(z^t), h(x(z^t, y^t))\rangle| \notag  \\
			\leq & M_h\|y^t - y_+^t(z^t)\| + \sqrt{m}\underline{B}K_h\kappa_3\|x^t - x^{t+1}\| + \frac{m\underline{B} \eta_t}{\alpha} \notag \\
			\leq & \left( M_h + \sqrt{m}\underline{B}K_h\kappa_3 + \frac{m\underline{B}(\kappa_3 + 1/\beta)}{\alpha}\right) \varrho_t = \lambda_3 \varrho_{t+1} , \label{eq: cs_error}
	\end{align}
	where the second inequality is due to $h$ being bounded and $K_h$-Lipschitz over $X$, \eqref{eb3}, and Lemma  \ref{lemma: primal feas and cs}. As a result of \eqref{eq: d_error}-\eqref{eq: cs_error}, we see that 
	\begin{align*}
		 & \max\{ \|\xi^t\|, \|\Pi_{X}(h(x^{t+1}))\|, |\langle y^t, h(x^{t+1}) \rangle|\}	\\
		 \leq & \max\{\lambda_1, \lambda_2, \lambda_3 \}\varrho_{t+1} \leq \max\{\lambda_1, \lambda_2, \lambda_3 \} \sqrt{\frac{\Phi^0 - \underline{f}}{T \lambda_0}} \leq \epsilon
	\end{align*}
	where the second inequality is due to \eqref{eq: bound max diff}, and the last inequality is due to the claimed upper bound on $T$ in \eqref{eq: T_ub}. 
\end{proof}

\section{Experiments} \label{sec: experiments}

Convergence behavior of the proposed smoothed proximal ALM (SProx-ALM for short) for a class of nonconvex quadratic programming (QP) problems, i.e., nonconvex quadratic loss function with convex quadratic constraints, is studied in this section. In particular, the considered QP problems take the following form
\begin{equation}\label{opt:numerQP}
\begin{aligned}
    \min_{x}\ &\frac{1}{2}x^\top Qx+r^\top x\\
    \textrm{s.t.}\ &\frac{1}{2}x^\top A_ix + b_i^\top x + c_i \leq 0,\ i=1,2,\ldots,m \\
    & \ell_j \leq x_j \leq u_j, \ j=1,2,\ldots,n
\end{aligned}
\end{equation}
where $Q\in\mathbb{R}^{n\times n}$ is a symmetric matrix with smallest eigenvalue $\lambda_{\min}(Q)\leq 0$, $r\in\mathbb{R}^n$, and for $i=1,2,\ldots, m$, $A_i\in\mathbb{R}^{n\times n}$ is a positive semi-definite matrix, $b_i\in\mathbb{R}^n$, and $c_i<0$ is a negative constant so that the Slater condition holds. The lower and upper bounds for each coordinate $j=1,2,\ldots,n$ are set as $\ell_j=-10$ and $ u_j=10$.

Different choices of problem size $n$ are considered, i.e., $n=50,100,200$, and $m$ is fixed to be $m=20$. Each entry of $Q$, $r$, $A_i$, and $b_i$ is independently sampled from standard Gaussian distribution and $\lambda_{\min}(Q)$ is set to $-0.1,-1$ and $-10$ by minusing a properly scaled identity matrix. The stationary gap defined in \eqref{eq: nlp stationary dual infeas}-\eqref{eq: nlp stationary cs} is adopted as the performance metric, correspond to dual feasibility, primal feasibility, and complimentary slackness, respectively. Another ALM based algorithm named iALM~\cite{LiXu2020almv2} is adopted as baseline for the comparison. The parameters of iALM are set by referring\cite[Algorithm 3]{LiXu2020almv2} with notations keeping the same, i.e., $\epsilon=10^{-4}$, $\beta_0=0.01$, $p=|\lambda_{\min}(Q)|$, $L_{\min}=\rho$, $\sigma=3$, $\gamma_1=2$, $\gamma_2=1.25$, The parameters in SProx-ALM are set as $p=3|\lambda_{\min}(Q)|$, $\alpha=0.01$, $\beta=0.05$, $c=0.01$, $B=10^4$.

The convergence curves of one trial for different problem sizes $n$ with $\rho=10$ are compared in Fig.\ref{fig:SProxALM}. It is worth noting that one iteration of iALM corresponds to one inner line search iteration, as it solves a subproblem with the same form as the one solved in SProx-ALM. The comparison in Fig.\ref{fig:SProxALM} clearly demonstrates that the proposed SProx-ALM algorithm exhibits significantly faster convergence compared to iALM. In Tables~\ref{tab: res1}-\ref{tab: res2}, the computational time, number of gradient evaluations, and function evaluations are compared for achieving a stationary gap of $10^{-5}$. It can be observed that the proposed SProx-ALM is much faster (five to ten times faster, depending on the setting) than iALM, as it requires significantly fewer gradient evaluations and does not need to evaluate the function value.
\begin{figure}[H]
    \centering
    \includegraphics[width=0.32\linewidth]{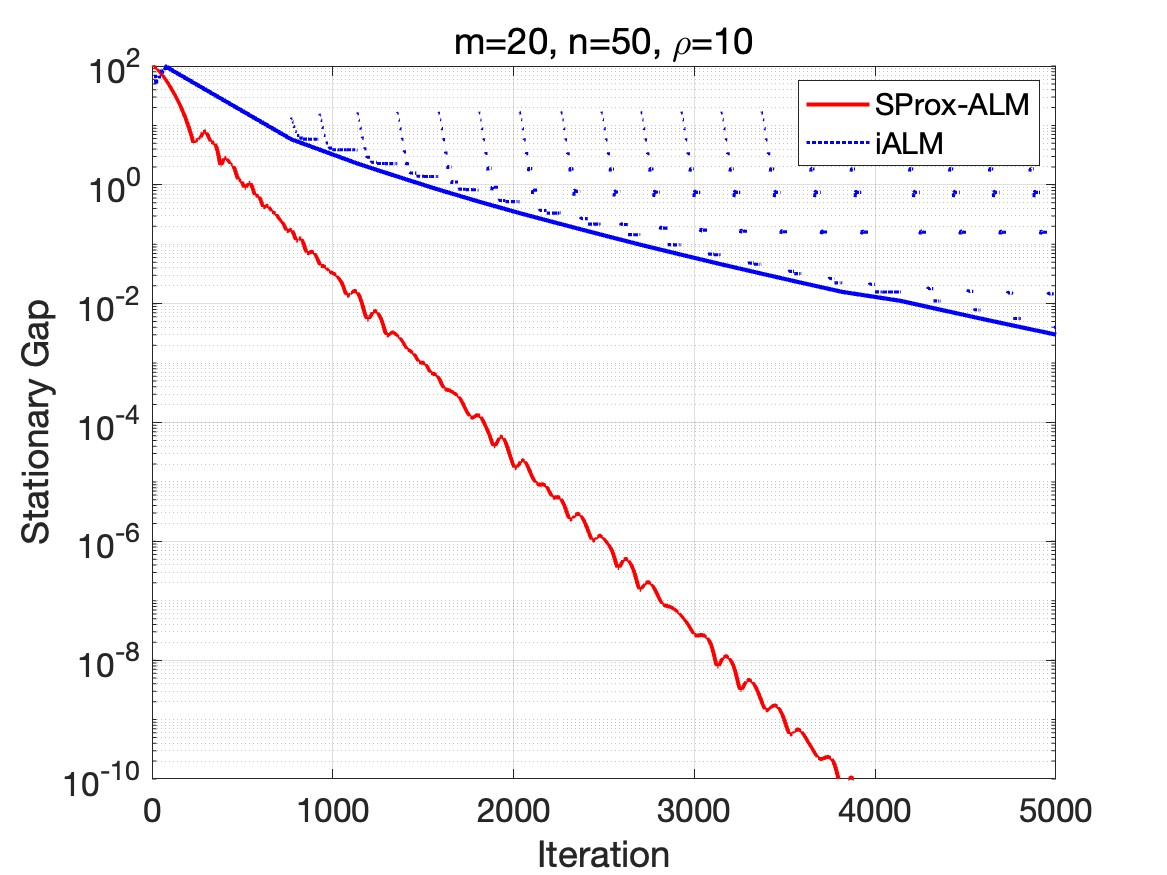}
    \includegraphics[width=0.32\linewidth]{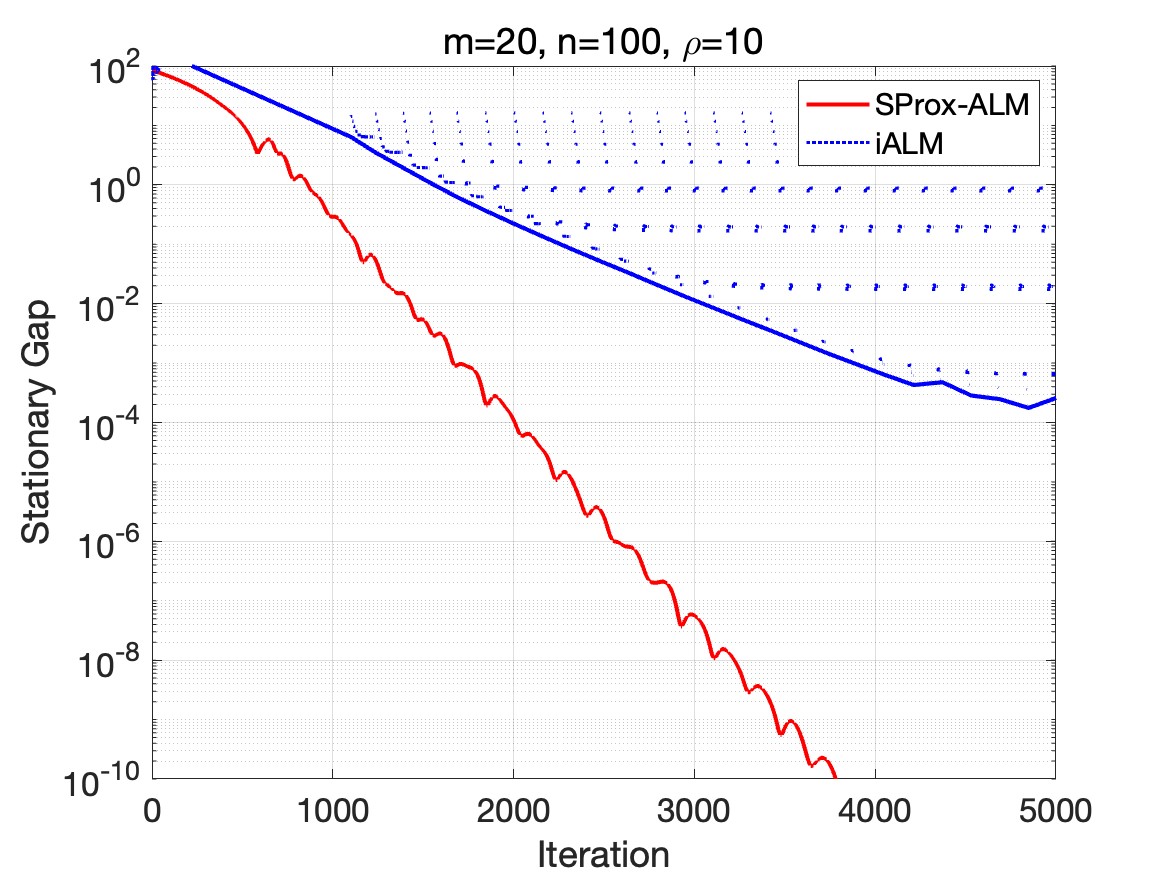}
    \includegraphics[width=0.32\linewidth]{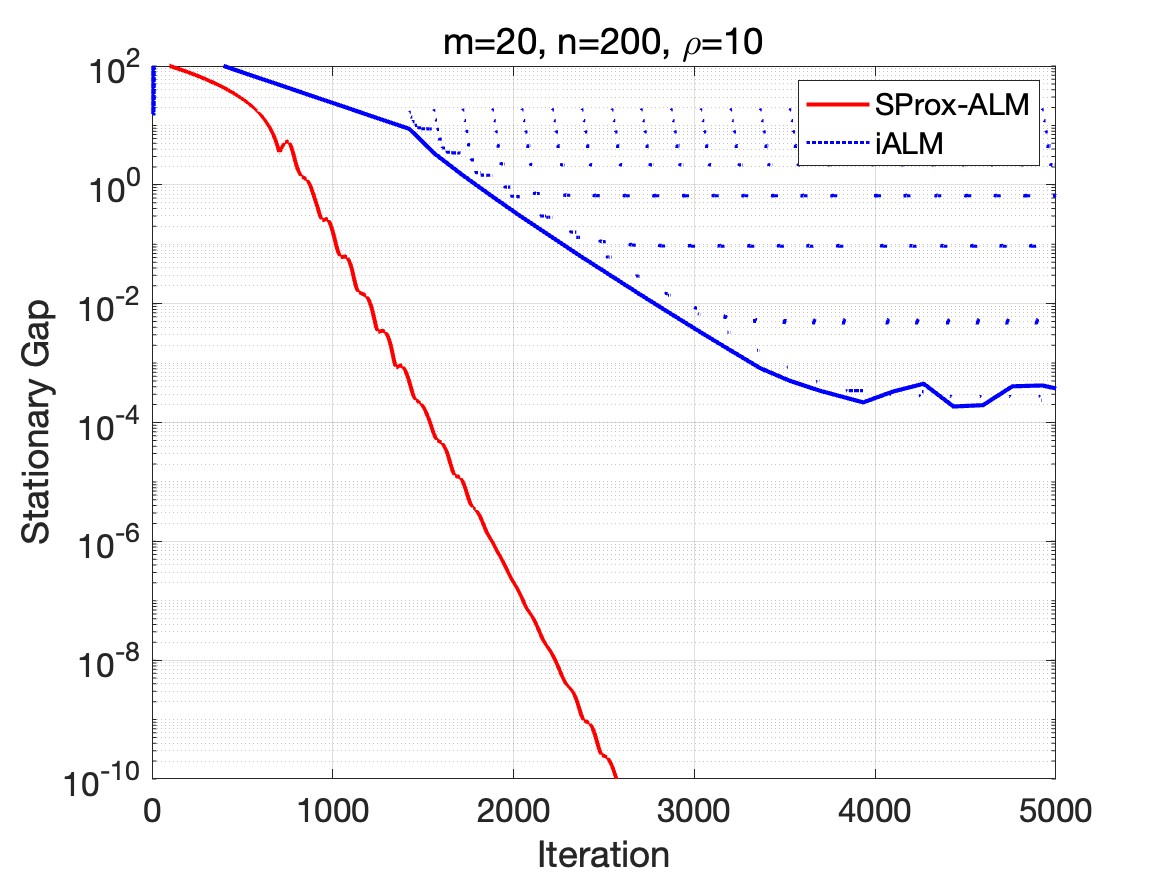}
    \caption{Convergence curves of iALM and SProx-ALM with $\lambda_{\min}(Q)=-10$.} 
    \label{fig:SProxALM}
\end{figure}

\begin{table}[H]
\centering
\small 
\caption{Results of problem instances with $\rho=0.1,\ n=50,\ m=20 $.}
\begin{tabular}{ccccc|cccc}
\toprule
& \multicolumn{4}{c}{\textbf{SProxALM}}  & \multicolumn{4}{c}{\textbf{iALM}}      \\ 
\midrule
Trial & Gap & Time & \# of Grad & \multicolumn{1}{c|}{\# of Obj} & \multicolumn{1}{c}{Gap} & \multicolumn{1}{c}{Time} & \multicolumn{1}{c}{\# of Grad} & \multicolumn{1}{c}{\# of Obj} \\ \hline
1  &  9.8e-06 & 0.458 & 5710 & 0& 1.3e-05 & 1.872 & 9159 & 11135 \\
2  &  9.9e-06 & 0.322 & 3984 & 0& 1.4e-05 & 1.839 & 8777 & 10688 \\
3  &  1.0e-05 & 0.463 & 5904 & 0& 2.3e-05 & 1.488 & 7307 & 8992 \\
4  &  1.0e-05 & 0.340 & 4180 & 0& 1.2e-05 & 1.606 & 8019 & 9797 \\
5  &  1.0e-05 & 0.473 & 5734 & 0& 2.3e-05 & 1.780 & 8948 & 10865 \\ 
\bottomrule
\end{tabular}
\label{tab: res1}
\end{table}

\begin{table}[H]
\centering
\small 
\caption{Results of problem instances with $\rho=1,\ n=50,\ m=20 $.}
\begin{tabular}{ccccc|cccc}
\toprule
& \multicolumn{4}{c}{\textbf{SProxALM}}  & \multicolumn{4}{c}{\textbf{iALM}}      \\ 
\midrule
Trial & Gap & Time & \# of Grad & \multicolumn{1}{c|}{\# of Obj} & \multicolumn{1}{c}{Gap} & \multicolumn{1}{c}{Time} & \multicolumn{1}{c}{\# of Grad} & \multicolumn{1}{c}{\# of Obj} \\ \hline
1  &  9.9e-06 & 0.372 & 4700 & 0& 4.6e-05 & 1.736 & 8367 & 10688 \\
2  &  1.0e-05 & 0.416 & 5266 & 0& 9.1e-05 & 1.859 & 9030 & 11519 \\
3  &  9.9e-06 & 0.336 & 4218 & 0& 9.2e-05 & 1.704 & 8264 & 10500 \\
4  &   9.9e-06 & 0.332 & 4152 & 0& 5.7e-05 & 1.521 & 7410 & 9481 \\
5  &   9.9e-06 & 0.375 & 4644 & 0& 8.8e-05 & 1.720 & 8443 & 10666 \\ 
\bottomrule
\end{tabular}
\end{table}

\begin{table}[H]
\centering
\small 
\caption{Results of problem instances with $\rho=10,\ n=50,\ m=20 $.}
\begin{tabular}{ccccc|cccc}
\toprule
& \multicolumn{4}{c}{\textbf{SProxALM}}  & \multicolumn{4}{c}{\textbf{iALM}}      \\ 
\midrule
Trial & Gap & Time & \# of Grad & \multicolumn{1}{c|}{\# of Obj} & \multicolumn{1}{c}{Gap} & \multicolumn{1}{c}{Time} & \multicolumn{1}{c}{\# of Grad} & \multicolumn{1}{c}{\# of Obj} \\ \hline
1  &  9.9e-06 & 0.406 & 5118 & 0& 7.8e-05 & 6.297 & 29401 & 39517 \\
2  &  1.0e-05 & 0.493 & 6164 & 0& 8.3e-05 & 7.451 & 36160 & 48045 \\
3  &  1.0e-05 & 0.559 & 7094 & 0& 9.2e-05 & 6.302 & 29929 & 41178 \\
4  &  1.0e-05 & 0.521 & 6652 & 0& 9.1e-05 & 5.987 & 27865 & 38756 \\
5  &  9.9e-06 & 0.447 & 5708 & 0& 9.9e-05 & 5.884 & 28127 & 37678 \\ 
\bottomrule
\end{tabular}
\end{table}

\begin{table}[H]
\centering
\small 
\caption{Results of problem instances with $\rho=0.1,\ n=100,\ m=20 $.}
\begin{tabular}{ccccc|cccc}
\toprule
& \multicolumn{4}{c}{\textbf{SProxALM}}  & \multicolumn{4}{c}{\textbf{iALM}}      \\ 
\midrule
Trial & Gap & Time & \# of Grad & \multicolumn{1}{c|}{\# of Obj} & \multicolumn{1}{c}{Gap} & \multicolumn{1}{c}{Time} & \multicolumn{1}{c}{\# of Grad} & \multicolumn{1}{c}{\# of Obj} \\ \hline
1  &  9.9e-06 & 0.376 & 3008 & 0& 7.3e-05 & 2.341 & 7798 & 9410 \\
2  &  9.9e-06 & 0.431 & 3518 & 0& 1.0e-05 & 3.483 & 11719 & 14104 \\
3  &  9.9e-06 & 0.427 & 3432 & 0& 1.3e-05 & 3.409 & 11555 & 13957 \\
4  &  9.8e-06 & 0.410 & 3270 & 0& 1.3e-05 & 3.472 & 11709 & 14132 \\
5  &  9.9e-06 & 0.412 & 3368 & 0& 1.6e-05 & 3.502 & 11718 & 14107 \\ 
\bottomrule
\end{tabular}
\end{table}

\begin{table}[H]
\centering
\small 
\caption{Results of problem instances with $\rho=1,\ n=100,\ m=20 $.}
\begin{tabular}{ccccc|cccc}
\toprule
& \multicolumn{4}{c}{\textbf{SProxALM}}  & \multicolumn{4}{c}{\textbf{iALM}}      \\ 
\midrule
Trial & Gap & Time & \# of Grad & \multicolumn{1}{c|}{\# of Obj} & \multicolumn{1}{c}{Gap} & \multicolumn{1}{c}{Time} & \multicolumn{1}{c}{\# of Grad} & \multicolumn{1}{c}{\# of Obj} \\ \hline
1  &  1.0e-05 & 0.392 & 3254 & 0& 3.6e-05 & 3.225 & 10589 & 13325 \\
2  &  1.0e-05 & 0.406 & 3362 & 0& 4.9e-05 & 2.607 & 8236 & 10553 \\
3  &  1.0e-05 & 0.378 & 3070 & 0& 4.3e-05 & 2.744 & 9014 & 11439 \\
4  &  1.0e-05 & 0.395 & 3256 & 0& 3.2e-05 & 2.754 & 9014 & 11431 \\
5  &  9.9e-06 & 0.418 & 3418 & 0& 5.4e-05 & 2.829 & 9241 & 11695 \\ 
\bottomrule
\end{tabular}
\end{table}

\begin{table}[H]
\centering
\small 
\caption{Results of problem instances with $\rho=10,\ n=100,\ m=20 $.}
\begin{tabular}{ccccc|cccc}
\toprule
& \multicolumn{4}{c}{\textbf{SProxALM}}  & \multicolumn{4}{c}{\textbf{iALM}}      \\ 
\midrule
Trial & Gap & Time & \# of Grad & \multicolumn{1}{c|}{\# of Obj} & \multicolumn{1}{c}{Gap} & \multicolumn{1}{c}{Time} & \multicolumn{1}{c}{\# of Grad} & \multicolumn{1}{c}{\# of Obj} \\ \hline
1  &  9.9e-06 & 0.478 & 3938 & 0& 8.4e-05 & 7.101 & 22535 & 30960 \\
2  &  1.0e-05 & 0.528 & 4372 & 0& 8.0e-05 & 8.886 & 28087 & 38088 \\
3  &  1.0e-05 & 0.501 & 4162 & 0& 7.5e-05 & 7.612 & 24096 & 32642 \\
4  &  1.0e-05 & 0.608 & 4926 & 0& 8.1e-05 & 9.589 & 30541 & 41483 \\
5  &  1.0e-05 & 0.994 & 8160 & 0& 1.0e-04 & 10.857 & 34706 & 46764 \\ 
\bottomrule
\end{tabular}
\end{table}

\begin{table}[H]
\centering
\small 
\caption{Results of problem instances with $\rho=0.1,\ n=200,\ m=20 $.}
\begin{tabular}{ccccc|cccc}
\toprule
& \multicolumn{4}{c}{\textbf{SProxALM}}  & \multicolumn{4}{c}{\textbf{iALM}}      \\ 
\midrule
Trial & Gap & Time & \# of Grad & \multicolumn{1}{c|}{\# of Obj} & \multicolumn{1}{c}{Gap} & \multicolumn{1}{c}{Time} & \multicolumn{1}{c}{\# of Grad} & \multicolumn{1}{c}{\# of Obj} \\ \hline
1  &  9.8e-06 & 2.434 & 2430 & 0& 7.5e-05 & 21.781 & 8955 & 10837 \\
2  &  9.7e-06 & 2.507 & 2510 & 0& 1.1e-05 & 22.788 & 9732 & 11872 \\
3  &  9.8e-06 & 2.540 & 2556 & 0& 1.5e-05 & 28.389 & 12122 & 14641 \\
4  &  1.0e-05 & 2.642 & 2546 & 0& 1.5e-05 & 29.461 & 11955 & 14459 \\
5  &  9.6e-06 & 2.389 & 2372 & 0& 7.4e-05 & 20.594 & 8718 & 10548 \\ 
\bottomrule
\end{tabular}
\end{table}

\begin{table}[H]
\centering
\small 
\caption{Results of problem instances with $\rho=1,\ n=200,\ m=20 $.}
\begin{tabular}{ccccc|cccc}
\toprule
& \multicolumn{4}{c}{\textbf{SProxALM}}  & \multicolumn{4}{c}{\textbf{iALM}}      \\ 
\midrule
Trial & Gap & Time & \# of Grad & \multicolumn{1}{c|}{\# of Obj} & \multicolumn{1}{c}{Gap} & \multicolumn{1}{c}{Time} & \multicolumn{1}{c}{\# of Grad} & \multicolumn{1}{c}{\# of Obj} \\ \hline
1  &  9.9e-06 & 2.297 & 2292 & 0& 8.8e-05 & 18.832 & 7843 & 10059 \\
2  &  9.9e-06 & 2.639 & 2666 & 0& 7.4e-05 & 17.840 & 7412 & 9525 \\
3  &  1.0e-05 & 2.496 & 2482 & 0& 6.7e-05 & 19.028 & 7752 & 9928 \\
4  &  9.8e-06 & 2.583 & 2524 & 0& 3.0e-05 & 23.440 & 9829 & 12543 \\
5  &  9.8e-06 & 2.368 & 2354 & 0& 4.6e-05 & 18.469 & 7498 & 9655 \\ 
\bottomrule
\end{tabular}
\end{table}

\begin{table}[H]
\centering
\small 
\caption{Results of problem instances with $\rho=10,\ n=200,\ m=20 $.}
\begin{tabular}{ccccc|cccc}
\toprule
& \multicolumn{4}{c}{\textbf{SProxALM}}  & \multicolumn{4}{c}{\textbf{iALM}}      \\ 
\midrule
Trial & Gap & Time & \# of Grad & \multicolumn{1}{c|}{\# of Obj} & \multicolumn{1}{c}{Gap} & \multicolumn{1}{c}{Time} & \multicolumn{1}{c}{\# of Grad} & \multicolumn{1}{c}{\# of Obj} \\ \hline
1  &  1.0e-05 & 3.508 & 3096 & 0& 9.7e-05 & 56.507 & 21199 & 28730 \\
2  &  1.0e-05 & 3.950 & 3166 & 0& 7.2e-05 & 59.468 & 23300 & 31192 \\
3  &  9.9e-06 & 3.653 & 3128 & 0& 7.5e-05 & 44.609 & 17858 & 24380 \\
4  &  1.0e-05 & 2.932 & 2888 & 0& 8.1e-05 & 45.309 & 17466 & 23859 \\
5  &  9.9e-06 & 3.364 & 3376 & 0& 9.0e-05 & 45.954 & 18218 & 25078 \\ 
\bottomrule
\end{tabular}
\label{tab: res2}
\end{table}

\section{Conclusion} \label{sec: conclusion}
This paper proposes a smoothed proximal ALM for minimizing a nonconvex smooth function over a convex domain subject to complicated convex functional constraints. The proposed method is single-looped and invokes only first-order oracles, and hence is easy to implement from a practical point of view. In addition, under mild regularity assumptions, we show that the proposed method achieves the best-known  iteration complexity of $\mathcal{O}(\epsilon^{-2})$ for the class of problems considered, complementing the existing literature on first-order methods for nonconvex constrained optimization. We have also numerically demonstrated the superiority of the proposed method over existing methods on various problems scales. 

{For future directions, we are mainly interested in developing a unified single-looped algorithmic framework that can handle both inequality and equality constraints while maintaining the $\mathcal{O}(\epsilon^{-2})$ iteration complexity.
Looking ahead, a natural extension would be to design algorithms that can handle both linear equality constraints \( Ax = b \) and non-linear inequalities \( h(x) \leq 0 \) simultaneously. An initial approach could be to define \( \bar{X} := \{x \in X : Ax = b\} \) and apply a modified version of Algorithm \ref{Alg:SProxALM} to problem \eqref{eq: nlp}, replacing \( X \) with \( \bar{X} \). This modification demands a more complex primal update, as shown in the following optimization problem:
\[
\min_{x \in X }\{ \|x - u\|^2 : Ax = b\}.
\]
While first-order augmented Lagrangian methods (ALM) or penalty approaches could address this, they typically require multiple loops, potentially detracting from the desired $\mathcal{O}(\epsilon^{-2})$ complexity. Despite ongoing research in this area by scholars like Lin et al. \cite{lin2019inexact} and Kong et al. \cite{kong2023iteration}, a single-looped solution remains elusive. This gap sets the stage for future research aimed at achieving efficient single-looped implementations for these more complex problem structures.}

\appendix

\section{Proof of Lemma \ref{lemma: potential function}}\label{appendix: potential function}
We first show that $\Phi^t$ is indeed bounded from below.
\begin{lemma}\label{lemma: lower bounded}
Suppose Assumption \ref{assumption: basic assumption} holds. For any $x,z\in X $ and $y \in Y$, we have $\Phi(x, y, z)\ge \underline{f}$.
\end{lemma}
\begin{proof}
By the definition of $d(\cdot)$ and $P(\cdot)$, we have
\begin{align}
K(x, z; y)\ge d(y, z), ~ P(z)\ge d(y, z), \text{~and~} P(z)\ge \underline{f}, \notag
\end{align}
which follows that $\Phi(x, y, z)\geq P(z) \geq \underline{f}$.
\end{proof}

Next, we present some basic \textit{error bounds} that can be derived from convexity of subproblems. 
\begin{lemma}\label{lemma: primal eb}
	Suppose Assumption \ref{assumption: basic assumption} holds.  Define constants 
	\begin{align} \notag 
		\kappa_1 := \frac{p}{\mu_K}, ~ \kappa_2 := \frac{K_h}{\mu_K}, \text{~and~} \kappa_3:=	1 + \frac{1}{c \mu_K}. 
	\end{align}
	Then for any $y, y' \in Y$ and $z, z' \in X$, it holds that 
\begin{align}
 	\|x(y, z)-x(y, z')\|\le & \kappa_1\|z-z'\|, \label{eb1}\\
	\|x(z)-x(z')\|\le &\kappa_1\|z-z'\|, \label{eb11}\\
	\|x(y, z)-x(y', z)\|\le & \kappa_2\|y-y'\|\label{eb2}.
\end{align}
In addition, for any $t \in \N$, we have 
\begin{align}	
	\|x^{t+1}-x(y^{t}, z^t)\|\le & \kappa_3\|x^t-x^{t+1}\|, \label{eb3} \\
    \|y^{t+1}-y^t_+(z^t)\|  \leq & \alpha K_h\kappa_3 \|x^t-x^{t+1}\|.\label{eb4}
\end{align}
\end{lemma}
\begin{proof}
The proofs of \eqref{eb1}, \eqref{eb11}, and \eqref{eb3} can be found in \cite[Lemma 3.10]{zhang2020proximal} and are hence omitted. We first prove \eqref{eb2}. By the strong convexity of $K(\cdot, z; y)$ in $x$, we have
\begin{align*}
	K(x(y', z), z; y)-K(x(y, z), z; y)   \ge & \frac{\mu_K}{2}\|x(y, z)-x(y', z)\|^2,\\
	K(x(y, z), z; y')-K(x(y', z), z; y') \ge & \frac{\mu_K}{2}\|x(y, z)-x(y', z)\|^2.
\end{align*}
By the linearity of $K(x, z;\cdot)$ in $y$, we have
\begin{align*}
	K(x(y, z), z; y')-K(x(y, z), z; y)   = &  \langle h(x(y, z)), y'-y\rangle, \\
	K(x(y', z), z; y)-K(x(y', z), z; y') = & \langle h(x(y', z)), y-y'\rangle.
\end{align*}
Combining the above (in)equalities, we have
\begin{align*}
\mu_K\|x(y, z)-x(y', z)\|^2\le\langle h(x(y, z))-h(x(y', z)), y'-y\rangle \le K_h \|x(y', z)-x(y, z)\|\|y-y'\|,
\end{align*}
where the last inequality uses the Cauchy-Schwarz inequality and the Lipschitz continuity of $h(x)$. Hence \eqref{eb2} holds with the claimed $\kappa_2$. It remains to prove \eqref{eb4}:
\begin{align*}
    \|y^{t+1}-y^t_+(z^t)\| =& \|\Pi_Y(y^t+\alpha h(x^{t+1})) - \Pi_Y(y^t+\alpha h(x(y^t, z^t)))\|\\
    \le&\alpha \|h (x^{t+1}) - h(x(y^t, z^t))\| \leq \alpha K_h\kappa_3 \|x^{t} - x^{t+1}\|. 
\end{align*} 
where the first inequality is due to the non-expansiveness of the projection operator, the second inequality is due to $h$ being $K_h$-Lipschitz, and the last inequality is due to \eqref{eb3}. 
\end{proof}

The following lemma immediately follows from Lemma \ref{lemma: primal eb}.
\begin{lemma}\label{Lipschitz-dual}
Suppose Assumption \ref{assumption: basic assumption} holds.  The dual function $d(y, z)$ is a Lipschitz differentiable in $y$ with modulus $K_h \kappa_2$, i.e., $\|\nabla_y d(y,z)-\nabla_y d(y',z)\|\le K_h \kappa_2 \|y-y'\|, \  \forall y, y'\in Y.$
\end{lemma}
\begin{proof}
By Danskin's theorem in convex analysis, we know that $d(y,z)$ is a differentiable function with $\nabla_yd(y, z)=\nabla_yK(x(y, z), z; y)=h(x(y,z))$. Moreover, for any $y, y'\in Y$, we have
\begin{align*}
	\|\nabla_y d(y,z) - \nabla_y d(y', z)\| = & \|h(x(y,z)) - h(x(y',z))\| \leq K_h\|x(y,z) - x(y', z)\| \leq K_h \kappa_2 \|y - y'\|,
\end{align*}
where the inequalities are due to $h$ being Lipschitz and \eqref{eb2}. 
\end{proof}

In the next three lemmas, we establish primal descent, dual ascent, and proximal descent properties of the proposed method. 
\begin{lemma}[Primal Descent] \label{lemma: primal descent}
Suppose Assumption \ref{assumption: basic assumption} holds, and choose $$0 < c \leq  \frac{1}{L_f + L_h(Y) + p}.$$ 
For any $t \in \N$, we have
\begin{align*} 
	K(x^t, z^t; y^t)-K(x^{t+1}, z^{t+1}; y^{t+1}) \geq  \frac{1}{2c}\|x^t-x^{t+1}\|^2+ \langle h(x^{t+1}), y^t - y^{t+1}\rangle +\frac{p}{2\beta}\|z^t-z^{t+1}\|^2.
\end{align*}
\end{lemma}
\begin{proof}
Notice that the step of updating $x$ is a standard gradient projection, hence we have
\begin{align}
	K(x^t, z^t; y^t)-K(x^{t+1}, z^t; y^t)\ge\frac{1}{2c}\|x^t-x^{t+1}\|^2. \notag 
\end{align}
Next, by definition of $\nabla_yK(x, z;y)$,   we have
\begin{align}
 K(x^{t+1}, z^t; y^t)-K(x^{t+1}, z^t; y^{t+1}) = \langle h(x^{t+1}), y^t - y^{t+1}\rangle. \notag 
\end{align}
Based on the update of variable $z^{t+1}$, i.e. $z^{t+1}=z^t+\beta(x^{t+1}-z^t)$, it is easy to show that
\begin{align}
K(x^{t+1}, z^t; y^{t+1})-K(x^{t+1}, z^{t+1}; y^{t+1})\ge \frac{p}{2\beta}\|z^t-z^{t+1}\|^2. \notag 
\end{align}
Combining the above three inequalities completes the proof. 
\end{proof}

\begin{lemma}[Dual Ascent]\label{lemma: dual ascent}
Suppose Assumption \ref{assumption: basic assumption} holds. For any $t \in \N$, we have
\begin{align*}
d(y^{t+1}, z^{t+1})-d(y^t, z^t) \ge &  \langle \nabla_yd(y^t, z^t), y^{t+1}-y^t\rangle-\frac{K_h \kappa_2 }{2}\|y^{t+1}-y^t\|^2 \notag \\
&+\frac{p}{2}\langle z^{t+1}-z^t, z^{t+1}+z^t-2x(y^{t+1}, z^{t+1})\rangle. 
\end{align*}
\end{lemma}
\begin{proof}
By the Lipschitz continuity of $d(y,z)$ in $y$ in Lemma \ref{Lipschitz-dual}, we have
\begin{align}
	d(y^{t+1},z^t) - d(y^t,z^t)&\geq \langle \nabla_y d(y^t,z^t),y^{t+1} - y^t \rangle-\frac{K_h \kappa_2 }{2}\|y^t-y^{t+1}\|^2. \notag 
\end{align}
Next, by definition of $K$, it holds that 
\begin{align*}
	d(y^{t+1},z^{t+1})-d(y^{t+1},z^t)&=K(x(y^{t+1},z^{t+1}),z^{t+1};y^{t+1})-K(x(y^{t+1},z^{t}),z^{t};y^{t+1})\\
		&\geq K(x(y^{t+1},z^{t+1}),z^{t+1};y^{t+1})-K(x(y^{t+1},z^{t+1}),z^{t};y^{t+1})\\
		&= \frac{p}{2}\|x(y^{t+1},z^{t+1})-z^{t+1}\|^2-\frac{p}{2}\|x(y^{t+1},z^{t+1})-z^{t}\|^2\\
		&= \frac{p}{2}\langle z^{t+1}-z^t, z^{t+1}+z^t-2x(y^{t+1}, z^{t+1})\rangle.	
\end{align*}
Combining the above two inequalities completes the proof. 
\end{proof}

\begin{lemma}[Proximal Descent]\label{lemma: proximal descent}
Suppose Assumption \ref{assumption: basic assumption} holds.  For any $t \in \N$, we have 
\begin{equation*}
P(z^{t})-P(z^{t+1})\ge \frac{p}{2}\langle z^t-z^{t+1}, z^t+z^{t+1}-2x({y}(z^{t+1}), z^t) \rangle
\end{equation*}
for any $y(z^{t+1}) \in Y(z^{t+1})$.
\end{lemma}
\begin{proof}
By Lemma \ref{lemma: minimax thm}, we have 
\begin{align*}
    P(z^{t+1}) = & \max_{y \in Y}d(y, z^{t+1})  = d({y}(z^{t+1}), z^{t+1}), \\ 
    P(z^t) = & \max_{y \in Y}d(y, z^t) \geq d({y}(z^{t+1}), z^t).
\end{align*}
As a result, it holds that
\begin{align*}
    P(z^{t+1})-P(z^t)& \leq d({y}(z^{t+1}), z^{t+1})-d({y}(z^{t+1}), z^t) \\
    & = K(x({y}(z^{t+1}), z^t), z^{t+1}; {y}(z^{t+1}))-K(x({y}(z^{t+1}), z^t), z^t; {y}(z^{t+1}))\\
    & =\frac{p}{2}\langle  z^{t+1}-z^t, z^{t+1}+z^t-2x({y}(z^{t+1}), z^t)\rangle.
\end{align*}
This completes the proof.
\end{proof}

We are now ready to prove Lemma \ref{lemma: potential function}. 
\begin{proof}[Proof of Lemma \ref{lemma: potential function}]
Combining Lemmas \ref{lemma: primal descent}-\ref{lemma: proximal descent}, we have 
\begin{align*}
    \Phi^t - \Phi^{t+1}
		\geq &  \frac{1}{2c}\|x^t-x^{t+1}\|^2+ \langle h(x^{t+1}),  y^t - y^{t+1} \rangle+\frac{p}{2\beta}\|z^t-z^{t+1}\|^2	\\
		& + 2\langle \nabla_yd(y^t, z^t), y^{t+1}-y^t\rangle-{K_h \kappa_2 }\|y^{t+1}-y^t\|^2\\
		& +{p}\langle z^{t+1}-z^t, z^{t+1}+z^t-2x(y^{t+1}, z^{t+1})\rangle \\
		&  -{p}\langle z^{t+1}-z^t, z^{t+1}+ z^t -2x({y}(z^{t+1}), z^t) \rangle \\
		= & \frac{1}{2c}\|x^t-x^{t+1}\|^2+ \langle h(x^{t+1}),  y^{t+1}-y^t \rangle+\frac{p}{2\beta}\|z^t-z^{t+1}\|^2	\\
		& + 2\langle \nabla_yd(y^t, z^t) - h(x^{t+1}), y^{t+1}-y^t\rangle-{K_h \kappa_2 }\|y^{t+1}-y^t\|^2\\
		& +2{p}\langle z^{t+1}-z^t,  x({y}(z^{t+1}), z^t) - x(y^{t+1}, z^{t+1}) \rangle. 
\end{align*}

Next we bound different terms above separately. First,  since $y^{t+1} = \Pi_Y(y^t + \alpha h(x^{t+1}))$, we have 
\begin{align*}
    \frac{1}{2}\|y^{t+1} - y^t - \alpha h(x^{t+1})\|^2 \leq \frac{1}{2}\|\alpha h(x^{t+1})\|^2 - \frac{1}{2}\|y^t - y^{t+1}\|^2.
\end{align*}
Expanding both sides of the above inequality gives
\begin{align*}
    \langle h(x^{t+1}), y^{t+1} - y^t\rangle \geq \frac{1}{\alpha} \|y^t - y^{t+1}\|^2. 
\end{align*}
Next, it holds that 
\begin{align*}
    2\langle \nabla_yd(y^t, z^t) - h(x^{t+1}), y^{t+1}-y^t\rangle  \ge & -2\|h(x(y^t, z^t)) - h(x^{t+1})\| \| y^{t+1}-y^t\| \\
		\ge & -2K_h \|x^{t+1} - x(y^t, z^t) \| \|y^t-y^{t+1}\|\\
        \geq & -2 K_h \left( \frac{1}{2\kappa_3^2} \|x^{t+1} - x(y^t, z^t) \|^2  + \frac{\kappa_3^2}{2}\|y^t-y^{t+1}\|^2 \right) \\ 
        \geq & -K_h \|x^t - x^{t+1}\|^2 - K_h\kappa_3^2\|y^t-y^{t+1}\|^2,
\end{align*}
where the first inequality is due to the Cauchy-Schwarz inequality, the second inequality is due to $h$ being $K_h$-Lipschitz, the third inequality is due to the AM-GM inequality, and the last inequality is due to \eqref{eb3}. Also notice that
\begin{align*}
    & 2p\langle z^{t+1}-z^t, x({y}(z^{t+1}), z^t)-x(y^{t+1}, z^{t+1})\rangle \\
  = &2p \langle z^{t+1}-z^t, x({y}(z^{t+1}), z^t)-x({y}(z^{t+1}), z^{t+1})\rangle +2p\langle z^{t+1}-z^t, x({y}(z^{t+1}), z^{t+1})-x(y^{t+1}, z^{t+1}) \rangle \\
\ge & -2p\kappa_1\|z^{t+1}-z^t\|^2+2p \langle z^{t+1}-z^t, x({y}(z^{t+1}), z^{t+1})-x(y^{t+1}, z^{t+1}) \rangle \\
\ge & -2p\kappa_1\|z^{t+1}-z^t\|^2-\frac{p}{6\beta}\|z^{t+1}-z^t\|^2-6p\beta\|x({y}(z^{t+1}), z^{t+1})-x(y^{t+1}, z^{t+1})\|^2, \\
 = & -2p\kappa_1\|z^{t+1}-z^t\|^2-\frac{p}{6\beta}\|z^{t+1}-z^t\|^2-6p\beta\|x(z^{t+1})-x(y^{t+1}, z^{t+1})\|^2,
\end{align*}
where the first inequality is due to the Cauchy-Schwarz inequality and \eqref{eb1}, the second inequality is due to the AM-GM inequality, and the last inequality is due to \eqref{eq: minimax2} in Lemma \ref{lemma: minimax thm}. Combining the previous three inequalities, we have
\begin{align}
    \Phi^t - \Phi^{t+1} \geq & \left( \frac{1}{2c} - K_h\right) \|x^t - x^{t+1}\|^2 +\left( \frac{1}{\alpha} - K_h\kappa_3^2 -K_h \kappa_2  \right) \|y^t - y^{t+1}\|^2  \notag \\
         & + \left( \frac{p}{3\beta} - 2p\kappa_1 \right) \|z^{t} - z^{t+1}\|^2 - 6p\beta\|x( z^{t+1})-x(y^{t+1}, z^{t+1})\|^2. \notag 
\end{align}
Since 
\begin{align*}
    \alpha \leq \frac{3}{4(K_h\kappa_3^2 -K_h \kappa_2 )}, ~\beta \leq \frac{1}{24\kappa_1}, \text{~and~} c \leq \frac{1}{4K_h}
\end{align*}
it is straightforward to check that 
\begin{align}
     \Phi^t - \Phi^{t+1} \geq &  \frac{1}{4c} \|x^t - x^{t+1}\|^2  + \frac{1}{4\alpha} \|y^t - y^{t+1}\|^2 + \frac{p}{4\beta}\|z^t -z^{t+1}\|^2 \notag \\
     & - 6p\beta\|x(z^{t+1})-x(y^{t+1}, z^{t+1})\|^2.  \label{eq: Phi_descent_1}
\end{align}
By \eqref{eb4}, we know $\|y^{t+1} - y^t_+(z^t)\| \leq \alpha K_h \sigma_3\|x^t - x^{t+1}\|$
\begin{align}\label{eq: diff_y_lb}
    \|y^t - y^{t+1}\|^2 = & \|y^t - y^t_+(z^t) + y^t_+(z^t) - y^{t+1} \|^2  \notag \\
     \geq & \frac{1}{2}\|y^t - y^t_+(z^t)\|^2 - \|y^t_+(z^t) - y^{t+1}\|^2 \notag \\
    \geq & \frac{1}{2}\|y^t - y^t_+(z^t)\|^2 - (\alpha K_h \sigma_3)^2\|x^t-x^{t+1}\|^2.
\end{align}
Using the error bounds derived in Lemma \ref{lemma: primal eb}, we have 
\begin{align}\label{eq: diff_x_ub}
    & \|x(z^{t+1})-x(y^{t+1}, z^{t+1})\|^2 \notag \\
	= & \|(x(z^{t+1})-x(z^{t}))+(x(z^t)-x(y^t_+(z^t), z^t))+(x(y^t_+(z^t), z^t)-x(y^{t+1}, z^t)) \\
 &+(x(y^{t+1}, z^t)-x(y^{t+1}, z^{t+1}))\|^2 \notag \\
	\le & 4\|x(z^{t+1})-x(z^{t})\|^2+4\|x(z^t)-x(y^t_+(z^t), z^t)\|^2+4\|x(y^t_+(z^t), z^t)\\
 &-x(y^{t+1}, z^t)\|^2+4\|x(y^{t+1}, z^t)-x(y^{t+1}, z^{t+1})\|^2 \notag \\
	\leq &  4\kappa_1^2\|z^t-z^{t+1}\|^2+4\|x(z^t)-x(y^t_+(z^t), z^t)\|^2 +4\kappa_2^2 \|y^t_+(z^t) -y^{t+1} \|+4\kappa_1^2\|z^t-z^{t+1}\|^2 \notag \\
	\leq & 8\kappa_1^2\|z^t-z^{t+1}\|^2+4\|x(z^t)-x(y^t_+(z^t), z^t)\|^2 + 4\kappa_2^2 \|x^t - x^{t+1}\|^2. 
\end{align}
Substituting \eqref{eq: diff_y_lb} and \eqref{eq: diff_x_ub} into \eqref{eq: Phi_descent_1}, we have 
\begin{align}
    \Phi^t - \Phi^{t+1} \geq & \left( \frac{1}{4c} 
    - \frac{\alpha^2 K_h^2 \kappa_3^2}{4\alpha} - 24 p\beta \kappa_2^2\right)\|x^t - x^{t+1}\|^2 +  \frac{1}{8\alpha} \|y^t - y_+^t(z^t)\|^2 \notag  \\
    & + \left( \frac{p}{4\beta} - 48 p\beta \kappa_1^2\right) \|z^t - z^{t+1}\|^2 - 24 p \beta \|x(z^t)-x(y^t_+(z^t), z^t)\|^2.  \label{eq: Phi_descent_2}
\end{align}
Since we also choose 
\begin{align}
    \alpha \leq \frac{1}{4cK_h^2\kappa_3^2}, ~\beta \leq \frac{1}{24\kappa_1},  \text{~and~} c\leq \frac{\kappa_1}{16 p \kappa_2^2}, \notag 
\end{align}
it holds that 
\begin{align}
    \frac{\alpha^2 K_h^2 \kappa_3^2}{4\alpha} +  24 p\beta \kappa_2^2 \leq \frac{1}{8c} \text{~and~} 48 \beta \kappa_1^2 \leq \frac{1}{8 \beta}.\notag 
\end{align}
As a result, \eqref{eq: Phi_descent_2} can be reduced to 
\begin{align}
    \Phi^t - \Phi^{t+1} \geq & \frac{1}{8c} \|x^t - x^{t+1}\|^2 + \frac{1}{8 \alpha} \|y^t - y_+^t(z^t)\|^2  +\frac{p}{8\beta} \|z^t - z^{t+1}\|^2 - 24 p \beta \|x(z^t)-x(y^t_+(z^t), z^t)\|^2. \notag 
\end{align}
This completes the proof.
\end{proof}

\section{Proof of Proposition \ref{prop: strong eb}} \label{sec: proof of strong eb}
First, we reduce the dual error bound to a perturbation bound.
Let 
\begin{align}\label{eq: perturbed min-max}
	x^*(r,z) 
             = & \argmin_{x\in X} f(x)+  \left( \max_{y \in Y}\langle y, h(x)-r\rangle \right) + \frac{p}{2}\|x-z\|^2, \\
    y^*(r,z) \in & \Argmax_{y \in Y} ~ \langle y, h(x^*(r,z) )-r\rangle.
\end{align}
The next lemma establishes the equivalence between $x(y_+(z), z)$ and $x^*(r, z)$ for some proper $r \in \R^m$. 
\begin{lemma}\label{lemma: perturbation r}
Given $(z, y) \in X\times Y$, let 
\begin{align}
\tilde{r} :=\frac{y_+(z)}{\alpha}+ h(x(y_+(z),z))- \frac{y}{\alpha}- h(x(y,z)). \label{eq: perturbation_r}
\end{align}
Then we have $x^*(\tilde{r},z)=x(y_+(z),z).$
\end{lemma}
\begin{proof}
	We claim that $x(y_+(z), z)$ and $y_+(z)$ are min-max solutions to \eqref{eq: perturbed min-max} with $r = \tilde{r}$. We first consider the optimality condition of the max-problem, which can be written as 
	\begin{align}
		\Pi_Y\Big( y + \alpha (h(x)-\tilde{r})\Big) = y	 \notag 
	\end{align}
	with a pair of min-max solution $(x,y)$ to \eqref{eq: perturbed min-max}, i.e., a gradient projection step should be a fixed-point iteration at optimal solution. Now take $x = x(y_+(z), z)$ and $y = y_+(z)$, the left-hand side of the above becomes 
	\begin{align}
		\Pi_Y\Big(y_+(z) + \alpha (h(x(y_+(z), z))-\tilde{r})\Big) = \Pi_Y\Big( y + \alpha h(x(y, z))\Big) = y_+(z), \notag 
	\end{align}
	where the first equality is due to the definition of $\tilde{r}$ and the second inequality is due to the definition of $y_+(z)$. Hence the pair $\left(x(y_+(z), z), y_+(z) \right)$ satisfies the optimality condition of the max-problem in \eqref{eq: perturbed min-max}. The optimality condition of $x(y_+(z), z)$ reads: 
	\begin{align}
		0 \in \nabla f(x(y_+(z), z)) + \nabla h(x(y_+(z), z))^\top y_+(z) + p(x(y_+(z), z) - z) + N_X(x(y_+(z), z)). \notag 	
	\end{align}
    So the pair also satisfies the optimality condition of the min-problem in \eqref{eq: perturbed min-max}.
\end{proof}

Next we show that if $B$ is chosen to be sufficiently large, then $x^*(r, z)$ is optimal to the constrained version of the min-max problem \eqref{eq: perturbed min-max}, where constraints $h(x) \leq 0$ are slightly perturbed. 

\begin{lemma}\label{lemma: equivalence between perturbed problems}
    Suppose Assumptions \ref{assumption: basic assumption} and \ref{assumption: convex regularity} holds. Let $\|r\| \leq \frac{\Delta_0}{4}$, $p > L_f$, and choose  $Y = [0, B]^m$ with
    \begin{align} \label{eq: B_underline_eb}
        B >  \frac{4\nabla_f D_X + 2 pD_X^2}{\Delta_0}. 
    \end{align}
    Then $x^*(r, z)$ is the optimal solution to the following problem: 
    \begin{align}\label{eq: perturbed problem}
        \min_{x \in X} \{ f(x) + \frac{p}{2}\|x-z\|^2~:~ h(x) \leq r\}. 
    \end{align}
    In addition, $y^*(r, z)$ are multipliers corresponding to $h(x) \leq r$, and $y_i^*(r, z)\in [0, B)$ for all $i\in [m]$. 
\end{lemma}
\begin{proof}
    The optimality of $x^*(r, z)$ reads: 
    \begin{align} \label{eq: perturbed primal stationarity}
        0 \in  \nabla f(x^*(r, z)) + \nabla h(x^*(r, z))^\top y^*(y, z) + N_X(x^*(r, z)).
    \end{align}
    Since $ y^*(r, z) \in \Argmax_{y \in Y} \langle h(x^*(r, z)) - r, y\rangle$, each $i \in [m]$ belongs to one of the following three sets: 
     \begin{align*}
        I_1 := & \{i \in [m]: h_i(x^*(r, z)) -r_i < 0,  ~y_i^*(r, z) = 0\}, \\
        I_2 := & \{i \in [m]: h_i(x^*(r, z)) -r_i  = 0, ~y_i^*(r, z) \in [0, B)\}, \\
        I_3 := & \{i \in [m]:h_i(x^*(r, z)) -r_i \geq 0, ~y_i^*(r, z) = B\}.
    \end{align*}
    We claim $I_3 = \emptyset$, and consequently, \eqref{eq: perturbed primal stationarity} and the fact that $I_1\cup I_2 = [m]$ imply that $x^*(r, z)$ is optimal to \eqref{eq: perturbed problem} and $y_i^*(r, z) \in [0, B)$ for all $i\in [m]$. Notice that \eqref{eq: perturbed primal stationarity} suggests that $x^*(r, z)$ is the unique optimal solution to the problem 
    \begin{align*}
        \min_{x \in X} f(x) + \langle y^*(r, z), h(x) - r\rangle + \frac{p}{2}\|x-z\|^2. 
    \end{align*}
    For the purpose of contradiction, suppose $I_3 \neq \emptyset$, and we shall construct a different point that achieves a lower objective than $x^*(r, z)$. By Assumption \ref{assumption: convex regularity}, there exists some $\theta \in (0, 1)$ such that $x_\theta := \theta \hat{x} + (1-\theta)x^*(r, z)$ satisfies that  $h_i(x_\theta) < -\frac{1}{2} \Delta_0$ for all $i \in [m]$. Notice that $|r_i|\leq \|r\| \leq \frac{\Delta_0}{4}$ for $i \in [m]$, and we have
    \begin{align}\notag 
        y_i^*(r, z) \times \left(h_i(x_\theta) - h_i(x^*(r, z))\right)
        \begin{cases}
            = 0, \quad & i \in I_1,\\
            \leq 0,  \quad & i \in I_2, \\
            < -\frac{\Delta_0 B}{4}, \quad & i \in I_3.
        \end{cases}
    \end{align}
    As a result, 
    \begin{align*}
       & \left( f(x_{\theta}) + \langle  y^*(r, z) , h(x_\theta) -r\rangle + \frac{p}{2}\|x_\theta - z\|^2 \right) \\
       &- \left( f(x^*(r, z)) + \langle  y^*(r, z) , h(x^*(r, z)) -r\rangle  + \frac{p}{2}\|x^*(r, z) - z\|^2\right) \\
       \leq & f(x_{\theta}) - f( y^*(r, z)) +  \frac{p}{2}\|x_\theta - z\|^2 + \sum_{i\in I_3}  y_i^*(r, z) \left(h_i(x_\theta) - h_i(x^*(r, z))\right) \\
       < & \nabla_f D_X + \frac{p}{2}D_X^2  -\frac{\Delta_0 B}{4} < 0, 
    \end{align*}
    where the second inequality is due to $f$ being $\nabla_f$ over $X$ and the compactness of $X$, and the third inequality is due to the lower bound of $B$ in \eqref{eq: B_underline_eb}. This is a desired contradiction. 
\end{proof}

Our subsequent analysis needs to use the fact that $x^*(r, z)$ is continuous in $(r, z)$, formally stated in the next lemma.
\begin{lemma}\label{lemma: continuous solution}
    The vector $x^*(r, z)$ is continuous in $(r, z)$. 
\end{lemma}
\begin{proof}
    To see the continuity with respect to $r$, fix $z$ and let $\{r^k\}_{k \in \N}$ be a sequence convergent to some $r^* \in \R^m$. By Lemma \ref{lemma: equivalence between perturbed problems}, we see that
    \begin{align*}
        & 0 \in \nabla f(x^*(r^k, z)) + p(x^*(r^k,z) - z) + \nabla h(x^*(r^\top,z))^\top y^*(r^k, z) + N_X(x^*(r^k,z)), \\
        & h(x^*(r^k, z))\leq r^k, \\
        & y^*_i(r^k, z) \in  [0, B), ~y^*_i(r^k, z) \times (h_i(x^*(r^k, z))-r^k_i) = 0, ~\forall i\in [m],
    \end{align*}
    Due to the boundedness of $\{(x^*(r^k, z),  y^*_i(r^k, z))\}_{k\in \N}$, we may assume that they converge to some $(x^*, y^*)\in X \times [0, B]^m$. By the closedness of normal cone and the fact that if $h_i(x^*)  < r_i^*$, then $h_i(x^*(r^k, z))  < r_i^k$ and hence $y_i^*(r^k_z) = 0$ for all sufficiently large $k$'s, taking limit on the above KKT condition suggests that $x^* = x^*(r^*, z)$ is the unique optimal solution to problem \eqref{eq: perturbed problem} with $r = r^*$. Hence we conclude that $x^*(\cdot, z)$ is continuous.  

     Using a similar proof to the primal error bound \eqref{eb11} in Lemma \ref{lemma: primal eb}, we see that $\|x^*(r, z) - x^*(r, z')\|\leq \kappa_1 \|z-z'\|$ for any fixed $r$. Consider a sequence $\{(r^k, z^k)\}_{k \in \N}$ convergent to some $(r^*, z^*)$, and let $\epsilon > 0$. There exists an index $T$ such that for all $t > T$ we have $\|z^t-z^*\| \leq \frac{\epsilon}{2\kappa_1}$ and hence $\|x^*(r^k, z^t) - x^*(r^k, z^*)\| < \frac{\epsilon}{2}$ for all $k \in \N$.  Since $x^*(\cdot, z^*)$ is continuous, we know $x^*(r^k, z^*) \rightarrow x^*(r^*, z^*)$, so there exists another index $K$ such that $\|x^*(r^k, z^*) - x^*(r^*, z^*)\| <  \frac{\epsilon}{2}$ for all $k > K$. As a result, for all $k > \max \{ K, T\}$, we have 
     \begin{align*}
        \|x^*(r^k, z^k) - x^*(r^*, z^*)\| \leq \|x^*(r^k, z^k) - x^*(r^k, z^*)\| + \|x^*(r^k, z^*) - x^*(r^*, z^*)\| < \frac{\epsilon}{2} + \frac{\epsilon}{2} = \epsilon.
     \end{align*}
    This proves the claim.    
\end{proof}

The next lemma summarizes some key properties of solutions to \eqref{eq: perturbed problem}. 
\begin{lemma}\label{eq: bounded hessian of submatrix}
    Suppose Assumptions \ref{assumption: basic assumption} and \ref{assumption: convex regularity} hold, and $B$ satisfies \eqref{eq: B_underline_eb}. There exists some $\hat{\delta}\in (0, \frac{\Delta_0}{4}]$ such that if $\max\{ \|x(z) - z\|, \|r\|\} \leq \hat{\delta}$, then 
    the following claims hold. 
    \begin{enumerate}
        \item For any $\mathcal{S}_A \subseteq A[x^*(r,z)]$, we have 
                \begin{align}
                    \sigma_{\min} ( J_{\mathcal{S}_A}(x^*(r, z))^\top ) \geq \frac{\sigma_{X^*}}{2}, \notag 
                \end{align}
                where $J_{\mathcal{S}_A}(x^*(r, z))$ is the submatrix of $J_A(x^*(r, z))$ consisting of rows in $\mathcal{S}_A$.
        \item Denote the active set of $x^*(r, z)$ in problem \eqref{eq: perturbed problem} as 
            $$A_r[x^*(r, z)]: = \{ i\in [m]:~ h_i(x^*(r, z)) = r_i\} \cup \{j \in [m+1:m+l]:~g_j(x^*(r, z)) = 0\}.$$
        Then we have $A_r[x^*(r, z)] \subseteq A[x^*(0, z)]$.
        \item There exist multipliers $\mu \in \R^l$ such that 
        \begin{subequations}\label{eq: perturbed kkt}
        \begin{align}
            & 0 = \nabla f(x^*(r,z)) + \nabla h(x^*(r, y))^\top y^*(r, z) + \nabla g(x^*(r, z))^\top \mu + p(x^*(r, z) - z),  \label{eq: perturbed kkt1}\\
            & \mu_j \geq 0, ~ \mu_j g_j(x^*(r, z)) = 0, ~\forall j \in [l],
        \end{align}
        \end{subequations}
        and 
         \begin{align} \label{eq: dual bound}
           \max\{\| y^*(r, z)\|, \|\mu\|\} \leq \sqrt{\|y^*(r, z)\|^2 +\|\mu\|^2} \leq  \Lambda: = \frac{2\nabla_f + 2pD_X}{\sigma_{X^*}}.
        \end{align}
    \end{enumerate}
\end{lemma}
\begin{proof}
    We prove the first two claims by contradiction, and the third claim follows accordingly. 
    \begin{enumerate}
        \item Suppose such a $\hat{\delta}$ does not exist, and there exists a sequence $\{(r^k, z^k)\}_{k \in \N}$ satisfying that $\|r^k\| \rightarrow 0$ and $\|z^k - x(z^k)\|\rightarrow 0$ as $k \rightarrow +\infty$, and the least singular value of a row submatrix $J^k$ of $J_{A}(x^*(r^k, z^k))$ drops below $\frac{\sigma_{X^*}}{2}$ as  $k\rightarrow +\infty$. Due to the compactness of $X$, passing to a subsequence if necessary, we assume that $(r^k, z^k)$ converges to $(0, z^*)$ for some $z^* \in X$. Notice that since $x^*(0, z) = x(z)$ for all $z\in X$, we have $x^*(0, z^*) = x(z^*) = z^*$ by Lemma \ref{lemma: continuous solution}. By Lemma \ref{lemma: equivalence between perturbed problems}, $z^*$ is a stationary solution to the original problem \eqref{eq: nlp}. Next, notice that $A[x^*(r^k, z^k)] \subseteq A[z^*]$ for sufficiently large $k$'s, so that $J^k$ will be closed to some row submatrix of $J_A(z^*)$, denoted by $J^k_*$, i.e., $\|J^k - J^k_*\|\rightarrow 0$ as $k \rightarrow +\infty$. Since the choice of submatrices is finite, we may assume that $J^k_*$ consists of the same subset of rows in $A[z^*]$. By the continuity of the singular value, we should have $\sigma_{\min} ((J^k)^\top) \geq \frac{\sigma_{X^*}}{2}$ for sufficiently large $k$'s, which is a desired contradiction. 

        \item Suppose such a $\hat{\delta}$ does not exist, and there exists a sequence $\{r^k\}_{k \in \N}$ satisfying that $\|r^k\| \rightarrow 0$, and for every $k \in \N$, there exists an index $i_k \in [m]$ such that $i_k \in A_r[x^*(r^k, z)]$ but $i_k \notin A[x^*(0, z)]$. Since the choice of $i_k \in [m+l]$ is finite, passing to a subsequence if necessary, we assume that $i_k$ stays the same for large enough $k \in \N$, simply denoted by $\hat{i}$. Since  $\hat{i} \in A_r[x^*(r, z)]$, we have $h_{\hat{i}}(x^*(r^k, z)) = r^k_i$ if $\hat{i} \in [m]$, or $g_{\hat{i}}(x^*(r^k, z)) = 0$ if $\hat{i} \in [m+1:m+l]$. By Lemma \ref{lemma: continuous solution} and the continuity of $h$ and $g$, we see that $h_{\hat{i}}(x^*(0, r)) = \lim_{k\rightarrow +\infty} h_{\hat{i}}(x^*(r^k, z)) = \lim_{k\rightarrow +\infty} r^k_i = 0$ and $g_{\hat{i}}(x^*(0, r)) = \lim_{k\rightarrow +\infty} g_{\hat{i}}(x^*(r^k, z)) = 0$, and hence $\hat{i} \in A[x^*(0, z)]$, which is a desired contradiction. 

        \item By parts 1 and 2, since active constraints defining $X$ have linearly independent gradients at $x^*(r,z)$, the existence of $\mu$ is ensured, and conditions \eqref{eq: perturbed kkt} are exactly the KKT condition for problem \eqref{eq: perturbed problem}. Notice that for $i \in [m]$ and $j \in [l]$ such that  $h_i(x^*(r,z)) < r_i$ and $g_j(x^*(r,z)) < 0$, we have $y^*_i(r, z) = \mu_j = 0$. We use $y_{A_r}$ and $\mu_{A_r}$ to denote subvectors of $y^*(r, z)$ and $\mu$ consisting of indices in $A_r[x^*(r,z)]$, then \eqref{eq: perturbed kkt1} is equivalent to 
        \begin{align*}
            0 = \nabla f(x^*(r,z)) + p(x^*(r, z) - z) + J_{A_r[x^*(r,z)]}(x^*(r, z))^\top [y_{A_r}^\top, \mu^\top_{A_r}]^\top. 
        \end{align*}
        By parts 1 and 2, we know the smallest singular value of $J_{A_r[x^*(r,z)]}(x^*(r, z))$ is bounded from below by $\sigma_{X^*}/2$, so the above equation implies that 
        \begin{align*}
             \max\{\| y^*(r, z)\|, \|\mu\|\} \leq & \sqrt{\|y_{A_r}\|^2 + \|\mu_{A_r}\|^2} \leq  \frac{2\nabla_f + 2pD_X}{\sigma_{X^*}}.
        \end{align*}
        This completes the proof. 
    \end{enumerate}
\end{proof}

Next we introduce the notion of basic set as follows. 
\begin{definition}
    Let $\hat{\delta}$ and $\Lambda$ be as in Lemma \ref{eq: bounded hessian of submatrix}. Denote 
    \begin{align}
        \Omega : = \{(z, r) \in \R^n \times \R^m:~ \max\{\|x(z) - z\|, \|r\|\} \leq \hat{\delta}\}
    \end{align}
    We say $\mathcal{B} \subseteq[m+l]$ is a basic set of $(r, z) \in \Omega$ if $J_{\mathcal{B}}(x^*(r, z))$ is of full row rank, and there exist some $y \in \R^{m}_+$ and $\mu \in \R^l_+$ such that 
    \begin{subequations} \label{eq: basic set}
    \begin{align}
        & 0 = \nabla f(x^*(r, z)) + J(x^*(r, z))^\top [y^\top, \mu^\top] ^\top + p(x^*(r, z) - z), \label{eq: basic set 1} \\
        & h_i(x^*(r, z)) = r_i,~ g_j(x^*(r, z)) = 0, ~\forall i, j \in \mathcal{B},  \label{eq: basic set 2}\\
        & y_i =\mu_j = 0, ~\forall i, j \notin \mathcal{B},  \label{eq: basic set 3}\\
        & \max\{\|y\|, \|\mu\| \} \leq \Lambda.
    \end{align}
    \end{subequations}
\end{definition}

Note that a basic set is similar to but slightly different from the active set: a basic set can be a subset of an active set, and the concept of basic set is associate with $(r, z) \in \Omega$. We first show basic sets are ``continuous" in the following sense.  
\begin{lemma}\label{lemma: continuity of basic sets}
Suppose $r^k \rightarrow r$ as $k\rightarrow +\infty$ and $\mathcal{B}$ is a common basic set of all $(r^k,z)$,  then $B$ is also a basic set of $(r,z)$.
\end{lemma}
\begin{proof}
By definition of the basic set, there exist $y^k \in \R^{m}_+$ and $\mu^k \in \R^l_+$ such that
\begin{align*}
    & \nabla f(x^*(r^k, z)) + J(x^*(r^k, z))^\top [(y^k)^\top, (\mu^k)^\top] ^\top + p(x^*(r^k, z) - z) \\
    & h_i(x^*(r^k,z)) = r^k_i,\ g_j(x^*(r^k,z))=0, ~\forall i,j\in \mathcal{B}\\
    & y_i^k=\mu_j^k=0, ~\forall i,j\notin \mathcal{B},\\
    & \max\{\|y^k\|, \|\mu^k\| \} \leq \Lambda.
\end{align*}
By the boundedness of $\|y^k\|$ and $\|\mu^k\|$, we let $(y,\mu)$ be a limit point of $\{(y^k,\mu^k)\}_{k \in \N}$. Passing to a subsequence if necessary, we may assume that $(r^k,y^k,\mu^k)\rightarrow (r,y,\mu)$ as $k\rightarrow+\infty$. Taking limits on the above system, we see that conditions \eqref{eq: basic set} are satisfied by the limit $(x^*(r, z), y, \mu)$. Moreover, $\mathcal{B}$ is a subset of $A[x^*(r,z)]$ and hence is a full rank matrix. Therefore, the result follows.
\end{proof}

We next show that $\|y^*(r,z)-y^*(r',z)\|$ can be bounded by $\|x^*(r,z)-x^*(r',z)\|$ if $r$ and $r'$ share a common basic set.
\begin{lemma}\label{label: control dual by primal}
Let $(r,z), (r',z)\in \Omega$. If $r$ and $r'$ share a common basic set, then 
\begin{align}
    \|y^*(r,z)-y^*(r',z)\|\le \frac{2\left(L_f + p + \Lambda \sqrt{L_h^2 + L_g^2} \right)}{\sigma_{X^*}} \|x^*(r,z)-x^*(r',z)\|.  \notag 
\end{align}
\end{lemma}
\begin{proof}
Let $(y, \mu)$ and $(y', \mu')$ be corresponding multipliers for $r$ and $r'$. We use $v = [y_\mathcal{B}^\top, \mu_{\mathcal{B}}^\top]^\top$ (resp. $v' = [(y'_\mathcal{B})^\top, (\mu'_{\mathcal{B}})^\top]^\top$) to denote subvectors of $(y, \mu)$ (resp. $(y, \mu)$) with components in the basic set $\mathcal{B}$. Also recall that $y_i = \mu_j = 0$ for $i, j\notin \mathcal{B}$. The optimality conditions of $x^*(r,z)$ and $x^*(r',z)$ read
\begin{align*}
    0 = & \nabla f(x^*(r,z)) + p(x^*(r,z) - z) + J_{\mathcal{B}}(x^*(r,z))^\top v,  \\
    0 = & \nabla f(x^*(r',z)) + p(x^*(r',z) - z) + J_{\mathcal{B}}(x^*(r',z))^\top v',
\end{align*}
so we have 
\begin{align}
    & \frac{\sigma_{X^*}}{2}\|v-v'\| \leq    \| J_{\mathcal{B}}(x^*(r,z))^\top(v-v')\| \notag \\
    \leq & \|\nabla f(x^*(r',z)) - \nabla f(x^*(r,z))\| + p\|x^*(r',z) - x^*(r,z)\| + \|(J_{\mathcal{B}}(x^*(r',z))- J_{\mathcal{B}}(x^*(r,z)))^\top v'\| \notag \\ 
    \leq & (L_f + p)\|x^*(r',z) - x^*(r,z)\| + \Lambda \|J_{\mathcal{B}}(x^*(r',z))- J_{\mathcal{B}}(x^*(r,z))\|.\notag 
\end{align}
In addition, it holds that
\begin{align*}
    & \|J_{\mathcal{B}}(x^*(r',z))- J_{\mathcal{B}}(x^*(r,z))\|^2  \leq  \|J_{\mathcal{B}}(x^*(r',z))- J_{\mathcal{B}}(x^*(r,z))\|_F^2  \\
    \leq &  \|J(x^*(r',z)) - J(x^*(r,z))\|_F^2 \\
    = & \sum_{i=1}^m \|\nabla h_i(x^*(r',z)) - \nabla h_i(x^*(r,z))\|^2 +  \sum_{j=1}^l \|\nabla g_j(x^*(r',z)) - \nabla g_j(x^*(r,z))\|^2  \\
    \leq & (L_h^2 + L_g^2) \|x^*(r',z) - x^*(r,z)\|^2. 
\end{align*}
Combining the above inequalities, we have 
\begin{align*}
    \|y^*(r,z)-y^*(r',z)\| \leq \|v-v'\| \leq \frac{2\left(L_f + p +  \Lambda \sqrt{L_h^2 + L_g^2} \right)}{\sigma_{X^*}}  \|x^*(r',z) - x^*(r,z)\|.
\end{align*}
This completes the proof. 
\end{proof}

\begin{lemma}\label{lemma: local eb}
For any $(r,z), (r',z) \in \Omega$ sharing a common basic set, we have
$$\|x^*(r,z)-x^*(r',z)\|\le \frac{2\left(L_f + p + \Lambda \sqrt{L_h^2 + L_g^2} \right)}{\mu_K \sigma_{X^*} }   \|r-r'\|. $$
\end{lemma}
\begin{proof}
    Since $x^*(r,z) = \argmin_{x\in X} f(x) + \langle y^*(r, z), h(x) - r \rangle + \frac{p}{2}\|x-z\|^2$, where the objective is strongly convex with modulus $\mu_K = p - L_f$, we have 
    \begin{align*}
        \frac{\mu_K}{2} \|x^*(r,z) - x^*(r',z)\|^2 \leq & f(x^*(r', z)) + \langle y^*(r, z), h(x^*(r',z)) - r \rangle + \frac{p}{2}\|x^*(r',z) - z\|^2  \\
        & - \left (f(x^*(r, z)) + \langle y^*(r, z), h(x^*(r,z)) - r \rangle + \frac{p}{2}\|x^*(r,z) - z\|^2  \right) \\
    \end{align*}
    similarly, the optimality of $x^*(r', z)$ gives
    \begin{align*}
        \frac{\mu_K}{2} \|x^*(r,z) - x^*(r',z)\|^2 \leq & f(x^*(r, z)) + \langle y^*(r', z), h(x^*(r,z)) - r' \rangle + \frac{p}{2}\|x^*(r,z) - z\|^2  \\
        & - \left (f(x^*(r', z)) + \langle y^*(r', z), h(x^*(r',z)) - r' \rangle + \frac{p}{2}\|x^*(r',z) - z\|^2  \right).
    \end{align*}
    Summing the above two inequalities, we have 
    \begin{align}
        \mu_K \|x^*(r,z) - x^*(r',z)\|^2  \leq & \langle  y^*(r,z) - y^*(r', z), h(x^*(r',z)) - h(x^*(r, z))\rangle \notag \\  
        =  &  \langle  y^*(r,z) - y^*(r', z), r' - r\rangle \leq  \|y^*(r,z) - y^*(r', z)\| \| r' - r\|,  \label{eq: bound primal by inner product}
    \end{align}
    where the equality is due to $h_i(x^*(r, z)) = r_i$ and $h_i(x^*(r', z)) = r'_i$ for all $i$ in the common basic set, and $y_i^*(r, z) = y_i^*(r', z) = 0$ for all $i$ not in the common basic set. Combining \eqref{eq: bound primal by inner product} with Lemma \ref{label: control dual by primal} proves the claim. 
\end{proof}

We perform a decomposition analysis as follows. 
\begin{proposition}\label{prop: global perturbation eb}
Let $\tilde{r}$ be given in Lemma \ref{lemma: perturbation r}, and suppose that $(\tilde{r}, z) \in \Omega$. There exist finite constants $0=\theta^0<\theta^1<\cdots <\theta^K=1$ such that $(\theta^k\tilde{r},z),(\theta^{k+1}\tilde{r},z)$ share a basic set for any $k\in \{0,\cdots,K-1\}$. Consequently, we have 
\begin{align}
    \|x^*(\tilde{r}, z) - x^*(0, z)\| \leq  \frac{2\left(L_f + p + \Lambda \sqrt{L_h^2 + L_g^2} \right)}{\mu_K \sigma_{X^*} } \|\tilde{r}\|. \notag 
\end{align}
\end{proposition}
\begin{proof}
We construct the constants recursively. Suppose that we already have $\theta^{k}$, and we shall prove the Existence of $\theta \in (\theta^k,1]$ sharing a common basic set with $\theta^{k}$. Let $\{\theta^k_j\}_{j \in \N}$ be a sequence convergent to $\theta^k$ from above. Since the choice of basic sets is finite, passing to a subsequence if necessary, we may assume that $\theta^k_j\tilde{r}$ share a common basic set $\mathcal{B}$ without loss of generality. By Lemma \ref{lemma: continuity of basic sets}, $\mathcal{B}$ is also a basic set of $\theta^k\tilde{r}$. Therefore, there exists some $\theta \in (\theta^k,1]$ such that $ \theta \tilde{r}$ shares a common active set with $\theta^k\tilde{r}$. Now take 
$\theta^{k+1}$ to be the supremum of $\lambda \in (\lambda^k, 1]$ such that $\lambda \tilde{r}$ shares a common basic set with $\theta^{k} \tilde{r}$. Again by Lemma \ref{lemma: continuity of basic sets}, $\lambda^{k+1}\tilde{r}$ and $\theta^k\tilde{r}$ share a common basic set. 

Next we show that $K$ is finite. In particular, we claim that any set $\mathcal{B}\subset [m+\ell]$ can be a basic set of $\theta^k \tilde{r}$ for at most two indices $ k \in \{0, 1, \cdots, K\}$: if $\mathcal{B}$ is a basic set of $\theta^{k_1} \tilde{r},\theta^{k_2}\tilde{r},\theta^{k_3}\tilde{r}$ with indices $k_1<k_2<k_3$, then $k_3>k_1+1$, contradicting to the definition of $\theta^{k_2}$. Therefore, $K$ is upper bounded by $2 \times 2^{(m+l)}$ and hence is finite. Finally, using Lemma \ref{lemma: local eb} and the triangle inequality complete the proof. 
\end{proof}

Finally we have all the pieces to prove the error bound in Proposition \ref{prop: strong eb}. 
\begin{proof}[Proof of Proposition \ref{prop: strong eb}]
For ease of notation, write $z^t = z$, $y^t = y$, $y^{t+1} = y_+(z)$, and $y_+^t(z^t) = y_+(z)$. Recall $\hat{\delta}$ from Lemma \ref{eq: bounded hessian of submatrix}. Then choose $\delta > 0$ such that 
\begin{align*}
\delta \leq \min \left\{ \frac{\hat{\delta}}{\frac{1}{\alpha} + K_h\kappa_2}, \hat{\delta} \right\},
\end{align*}
so that $\|z - x(z)\| \leq \hat{\delta}$ and 
\begin{align}
    \| \tilde{r}\| \leq & \frac{1}{\alpha} \| y_+(z) -y\| + \|h(x(y,z)) - h(x(y_+(z), z))\|\leq \left(  \frac{1}{\alpha} + K_h\kappa_2 \right) \| y_+(z) -y\| \leq \hat{\delta}, \notag 
\end{align}
where $\tilde{r}$ is  defined in \eqref{eq: perturbation_r}, and the second inequality is due to $h$ being $K_h$-Lipschitz and $x(\cdot, z)$  being $\kappa_2$-Lipschitz.  Then we have 
\begin{align*}
     \|x(y_+(z)) - x(z)\| 
    = & \|x^*(\tilde{r}, z) - x^*(0, z)\| \\ 
    \leq & \frac{2\left(L_f + p + \Lambda \sqrt{L_h^2 + L_g^2} \right)}{ \mu_K \sigma_{X^*} } \|\tilde{r}\| \\
    \leq & \frac{2\left(L_f + p + \Lambda \sqrt{L_h^2 + L_g^2} \right) \left(\frac{1}{\alpha} + K_h \kappa_2 \right) }{\mu_K \sigma_{X^*}} \|y - y_+(z)\|, 
\end{align*}
where the equality is due to Lemma \ref{lemma: perturbation r}, the first inequality is due to Proposition \ref{prop: global perturbation eb}, and the second inequality is due to the definition of $\tilde{r}$ in \eqref{eq: perturbation_r}. This completes the proof. 

\end{proof}

\bibliographystyle{plain}
\bibliography{paper_ref.bib}

\end{document}